\documentclass{amsart}
\usepackage{latexsym, graphicx, amssymb, cite, hyperref, tikz, tikz-cd, framed, mathrsfs}

\newtheorem{theorem}{Theorem}[section] % Comment out [section] if you want simple theorem numbering.
\newtheorem{lemma}[theorem]{Lemma}
\newtheorem{proposition}[theorem]{Proposition}
\newtheorem{corollary}[theorem]{Corollary}

\newtheorem{question}[theorem]{Question}
\newtheorem*{theorem*}{Theorem}
\newtheorem*{lemma*}{Lemma}
\newtheorem*{proposition*}{Proposition}
\newtheorem*{corollary*}{Corollary}
\newtheorem*{conjecture*}{Conjecture}
\newtheorem*{question*}{Question}

\theoremstyle{definition}
\newtheorem{definition}[theorem]{Definition}
\newtheorem{example}[theorem]{Example}

\newtheorem*{definition*}{Definition}
\newtheorem*{example*}{Example}
\newtheorem*{exercise*}{Exercise}

\theoremstyle{remark}

\newtheorem*{remark*}{Remark}

\numberwithin{equation}{section} % Comment out this line if you want simple equation numbering.

\newcommand{\R}{{\mathbb R}}

\newcommand{\Z}{{\mathbb Z}}

\newcommand{\N}{{\mathbb N}}

\newcommand{\1}{{\mathbf 1}}

\DeclareMathOperator{\sgn}{sgn}

\DeclareMathOperator{\Isom}{Isom}
\DeclareMathOperator{\interior}{int}

\usetikzlibrary{angles,arrows.meta,decorations.markings,decorations.pathreplacing}

\title{(Not) hearing where certain triangular drums are struck}
\author[T. Dixon]{Tyler Dixon}
\author[A. Miah]{Alif Miah}
\author[N. Safronov]{Nikita Safronov}
\author[S. Snyder]{Sean Snyder}
\author[A. Vakaryuk]{Anatolii Vakaryuk}

\author[E. Wyman]{Emmett Wyman}
\address{Department of Mathematics and Statistics, Binghamton University, Vestal NY}
\email{ewyman@binghamton.edu}
\thanks{This project is supported in part by NSF Grant DMS-2204397.} % Funding information

\begin{document}

\begin{abstract}
    We investigate whether or not one can hear at what point (up to symmetry) a drum is struck for some special planar triangles. We prove that one can hear where the equilateral and isosceles right triangles are struck. We also prove that the 30-60-90 triangle, surprisingly, has two audibly indistinguishable points. This is the first known topologically connected such example.
\end{abstract}

\maketitle

\section{Introduction}\label{sec: introduction}

\subsection{Background}

In his celebrated 1966 paper \cite{kac}, Kac asks, \emph{Can one hear the shape of a drum?} Put a little more precisely, does the Laplace spectrum uniquely determine the isometry class of a planar domain? This problem and its investigation has a long and celebrated history, and at the same time, there is much about it we still do not know. There are plenty of results, positive and negative, that together paint a nuanced picture of the problem. In his original paper, Kac used the isoperimetric inequality to prove disks are spectrally unique \cite{kac}. But later, Gordon, Webb, and Wolpert exhibited two isospectral, nonisometric polygonal planar domains \cite{gordon_webb_wolpert}. More recently, Grieser and Maronna showed that you can hear the shape of a triangular drumhead, provided you know your drumhead is some triangle to begin with \cite{grieser}. (See also Zelditch's survey \cite{zelditch_survey_2004} on the problem and the references therein).

Here, we investigate a different yet closely related problem posed by the last author and Xi \cite{echolocation}: \emph{Can one hear where a drum is struck?} In plain terms, if a drum of known shape is struck at a point, can one determine where it was struck up to symmetry? There is a small catalog of positive examples and few known negative examples, all of the latter being topologically disconnected.

The known positive planar examples include disks, squares, and rectangles with Dirichlet boundary conditions (see \cite{echolocation}, or see \cite{ugthesis} for a cleaner exposition with a wider range of boundary conditions). The initial objective of this paper was to expand this catalog to a few other triangles, namely the equilateral, 45-45-90, and 30-60-90 triangle with Dirichlet boundary conditions, and in the process we discovered there is a pair of audibly indistinguishable points on the 30-60-90 triangle, yielding the first topologically connected negative answer to the question, \emph{can one hear where a drum is struck?}

\subsection{Statement of Results}

We start by precisely describing the question. We summarize the definitions and a few key facts from \cite{echolocation}, but only for planar domains. See \cite{evansPartialDifferentialEquations2022} for general background on elliptic operators, their spectra and eigenfunctions.

Let $\Omega$ be a bounded open subset of $\R^2$ with piecewise-smooth boundary. A (Dirichlet) eigenfunction $e$ of the Laplace operator $\Delta$ on $\Omega$, with eigenvalue $-\lambda^2$ satisfies the boundary-value problem
\[
    \begin{cases}
        \Delta e(x) = -\lambda^2 e(x) & \text{ for } x \in \Omega, \\
        e(x) = 0 & \text{ for } x \in \partial \Omega.
    \end{cases}
\]
The Hilbert space $L^2(\Omega)$ admits an orthonormal basis $\{e_j : j \in \N\}$ of Dirichlet Laplace eigenfunctions with respective eigenvalues $\{-\lambda_j^2 : j \in \N\}$, where
\[
    0 < \lambda_1 \leq \lambda_2 \leq \ldots \qquad \text{ and } \qquad \lambda_j \to \infty.
\]
The \emph{local Weyl counting function} on $\Omega$ is defined as
\begin{equation}\label{def: local Weyl counting function}
    N(x,\lambda) = \sum_{\lambda_j \leq \lambda} |e_j(x)|^2 \qquad \text{ for $x \in \Omega$}
\end{equation}
and does not depend on the choice of eigenbasis. 

Recall, an \emph{isometry} on $\Omega$ is a bijective mapping $\phi : \Omega \to \Omega$ that preserves distances, i.e.
\[
    |\phi(x) - \phi(y)| = |x - y| \qquad \text{ for each } x,y \in \Omega.
\]
By triangulating distances from any three non-collinear points in $\Omega$, every isometry $\Omega \to \Omega$ extends uniquely to an isometry $\R^2 \to \R^2$, and hence has the form
\[
    \phi(x) = x_0 + Ux,
\]
where $x_0 \in \R^2$ and $U$ is a $2 \times 2$ orthogonal matrix (i.e. having real entries and satisfying $U^tU = I$). From this, one can show that the Laplacian commutes with isometries in the sense that 
\[
    (\Delta f) \circ \phi = \Delta (f \circ \phi) \qquad \text{ for each } f \in C^2(\Omega).
\]
A short exercise shows that if $\phi : \Omega \to \Omega$ is an isometry, then
\[
    N(\phi(x), \lambda) = N(x,\lambda) \qquad \text{ for each } x \in \Omega, \ \lambda \geq 0.
\]
Our main question essentially asks if the converse holds.

\begin{question} \label{question: echolocation}
    Let $N$ be the local Weyl counting function for a bounded planar domain $\Omega$ with piecewise-smooth boundary. If $x, y \in \Omega$ are two points for which
    \[
        N(x,\lambda) = N(y,\lambda) \qquad \text{ for each } \lambda \geq 0,
    \]
    must there be an isometry $\phi : \Omega \to \Omega$ for which $\phi(x) = y$?
\end{question}

\begin{theorem}\label{thm: main 1}
    The answer to Question \ref{question: echolocation} is `yes' for the equilateral and isosceles right triangles.
\end{theorem}

The next theorem says the answer to Question \ref{question: echolocation} is `no' for the 30-60-90 triangle. This is the first known topologically connected negative example. Note the 30-60-90 triangle has trivial symmetry group, so it suffices to identify any pair of points at which the local Weyl counting functions are identical.

\begin{theorem}\label{thm: main 2}
    Let $\Omega$ be the 30-60-90 triangle with vertices at $(0,0)$, $(\frac {\sqrt 3} 2, 0)$, and $(\frac {\sqrt 3} 2, \frac 1 2)$ as in Figure \ref{fig: audibly indistinguishable points}.
    Then,
    \[
        N(p, \lambda) = N(q, \lambda) \qquad \text{ for all $\lambda$} 
    \]
    where
    \[
        p = (\tfrac{7\sqrt 3} {16}, \tfrac 3 {16} ) \qquad \text{ and } \qquad q = (\tfrac{5\sqrt 3} {16}, \tfrac 3 {16} ).
    \]
    %Furthermore, $(p,q)$ is the only audibly indistinguishable pair of points in $\Omega$.
\end{theorem}

\begin{figure}
        \begin{tikzpicture}[scale=6]
          % vertices of the 30-60-90 triangle
          \coordinate (A) at (0,0);
          \coordinate (B) at ({sqrt(3)/2}, 0);
          \coordinate (C) at ({sqrt(3)/2}, 1/2);
        
          % the triangle
          \draw[thick] (A) -- (B) -- (C) -- cycle;
        
          % the two distinguished points
          \fill ({5*sqrt(3)/16}, {3/16}) circle (0.4pt);
          \fill ({7*sqrt(3)/16}, {3/16}) circle (0.4pt);
        \end{tikzpicture}
    \caption{The 30-60-90 triangle and the pair of audibly indistinguishable points of Theorem \ref{thm: main 2}. These points are not related by a symmetry.}
    \label{fig: audibly indistinguishable points}
\end{figure}
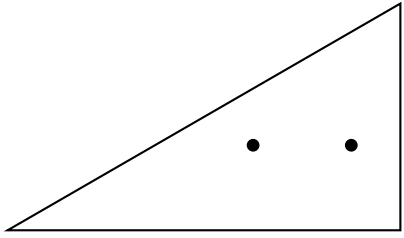

Each of these three model triangles share a common property: they ``unfold" along their edges to periodically tile the plane in a nice way. This property is what admits direct calculation of their Dirichlet-Laplace eigenfunctions, and also what allows us to cleanly characterize the local Weyl counting function $N(x, \cdot )$ in terms of billiard trajectories starting and returning to $x$. We use both perspectives in the proofs of the theorems.

\subsection{Structure of the paper} Section \ref{sec: nice domains} characterizes and classifies the planar domains with the nice ``unfolding" property. Section \ref{sec: billiards} establishes the audibility of certain billiard trajectories in such domains, and Section \ref{sec: Fourier series} builds the tools needed to write down their explicit Laplace eigenbases with help from the review of multidimensional Fourier series in Appendix \ref{app: lattices and fourier series}. Sections \ref{sec: proof of main theorem} and \ref{sec: proof of surprising theorem} prove Theorems \ref{thm: main 1} and\ref{thm: main 2}, respectively, with detailed computations relegated to Appendix \ref{sec: trig computations}.

\subsection{Acknowledgments} The authors are grateful to Yakun Xi for helpful conversations and for his insights into the problem. % and for pointing out that the pair of audibly indistinguishable points in the 30-60-90 triangle is unique.

\section{Domains that tile the plane evenly} \label{sec: nice domains}

We begin by making our unfolding definition precise.

\begin{definition}\label{def: tile evenly} A polygonal domain $\Omega$ in $\R^2$ is said to tile the plane \emph{evenly} if:
\begin{enumerate}
    \item $\Omega$ and the regions obtained by successive reflections over its sides create a periodic tiling of $\R^2$, and
    \item every vertex of the tiling is incident to an even number of tiles.
\end{enumerate}
\end{definition}

There are only a few such connected domains.

\begin{proposition}\label{prop: only a few tilings}
    The only connected domains that tile the plane evenly are rectangles, the $45$-$45$-$90$ triangle, the $30$-$60$-$90$ triangle, and the equilateral triangle.
\end{proposition}

\begin{proof}

\par We present a geometric argument to determine which polygons can tile the plane evenly by reflection. Suppose $\Omega$ is a connected domain that tiles the plane evenly. Then $\Omega$ must be a polygon whose interior angles allow for full rotation around each vertex. So each interior angle must divide $2\pi$ evenly.

\par First, consider the case where $\Omega$ is a quadrilateral. For a polygon to tile evenly by reflections, no interior angle can exceed $\frac{\pi}{2}$, since angles larger than $\frac{\pi}{2}$ would prevent multiple copies from fitting around a point without overlap. A general polygon with $n > 4$ sides has an interior angle sum of $(n - 2)\pi \geq 3\pi$, so at least one of its angles must be greater than or equal to $\frac{3\pi}{5} > \frac{\pi}{2}$. Hence, the maximum number of sides a tile may have in this context is 4. Furthermore, for a quadrilateral to tile evenly, it must be a rectangle. If any angle were smaller than $\frac{\pi}{2}$, another angle would have to be larger, a contradiction.

\par Now consider the case where $\Omega$ is a triangle, with interior angles $\alpha$, $\beta$, and $\gamma$. For such a triangle to tile the plane evenly by reflection, some number of copies of the triangle must fit around each vertex. That is, there must exist positive integers $m, n, p \in \mathbb{N}$ such that:
\begin{equation*}
    m\alpha = n\beta = p\gamma = \pi \quad \text{and} \quad \alpha + \beta + \gamma = \pi.
\end{equation*}
Combining these gives the condition:
\begin{equation*}
    \frac{1}{m} + \frac{1}{n} + \frac{1}{p} = 1.
\end{equation*}
Without loss of generality, assume $m \leq n \leq p$. Solving this equation yields exactly three integer solutions:
\begin{enumerate}
    \item $m = 2$, $n = 3$, $p = 6$,
    \item $m = 2$, $n = 4$, $p = 4$,
    \item $m = 3$, $n = 3$, $p = 3$.
\end{enumerate}
These correspond to the following triangles, respectively:
\begin{enumerate}
    \item The 30-60-90 triangle,
    \item The isosceles right triangle (45-45-90),
    \item The equilateral triangle.
\end{enumerate}
Thus, these are the only triangular tiles that can evenly tile the plane via reflection.
\end{proof}

Domains satisfying Definition \ref{def: tile evenly} satisfy some favorable algebraic properties.

\begin{proposition}\label{prop: characterization of evenly tiling domains}
    A compact domain $\Omega$ tiles the plane evenly by reflections over its sides if and only if $\Omega$ is a fundamental domain of the action of the group $R$ generated by reflections over the sides of $\Omega$.
\end{proposition}

\begin{proof}
Suppose $\Omega$ is a fundamental domain for the action of the group $R$ generated by reflections over the sides of $\Omega$. Since $\Omega$ is a fundamental domain, the orbits of $R$ acting on $\Omega$ cover all of $\mathbb{R}^2$. That is, the union of the images of $\Omega$ under $R$ yields a tiling of the plane.

Suppose two tiles overlapped. Then two distinct elements of $R$ would map $\Omega$ to regions that intersect, meaning the corresponding orbits intersect, which contradicts the definition of a fundamental domain (where each orbit meets $\Omega$ in exactly one point). Therefore, the tiles are disjoint except along their boundaries.

Furthermore, since the group $R$ is generated by reflections over the sides of $\Omega$, the tiling is edge-to-edge: adjacent tiles share full sides, not just parts of them. Such edge-to-edge tilings of the plane are classified (see \cite{KirbyUmble}), and all are periodic. Thus, the tiling obtained is periodic.

Now suppose there exists a vertex in the tiling that is incident to an odd number of tiles. Without loss of generality, pick one such tile and reflect it sequentially around the sides that meet at that vertex until the tile maps back onto itself. This sequence of reflections involves an odd number of reflections, which results in an orientation-reversing transformation.

However, this contradicts the fact that all group elements mapping $\Omega$ to tiles must preserve the uniqueness of orbits: an orientation-reversing transformation would map some interior point of $\Omega$ to another interior point of $\Omega$, implying two distinct group elements map $\Omega$ to overlapping regions. This again contradicts the fact that $\Omega$ is a fundamental domain. 

Therefore, no vertex can be incident to an odd number of tiles.

\begin{figure}[h]
    \centering
    \begin{tikzpicture}[scale=1.5]
        % Draw square grid
        \foreach \i in {0,...,3}{
            \draw[thick] (\i, 3) -- (\i, 0);
            \draw[thick] (0, \i) -- (3, \i);
        }

        % Draw inner lines
        \draw[thick] (1,0) -- (3,2) -- (2,3) -- (0,1) -- cycle;
        \draw[thick] (0,3) -- (3,0);

        % Draw highlighted region
        \draw[red, very thick,fill=red!5] (1,1) --  (2,1) -- (1,2) -- cycle;
        \draw[red, very thick, dotted] (1,1) -- (1.5, 1.5);
    \end{tikzpicture}
    \caption{Showcasing 45-45-90 triangle tiling and internal symmetries in red}
    \label{fig:lemma 2.1.1}
\end{figure}
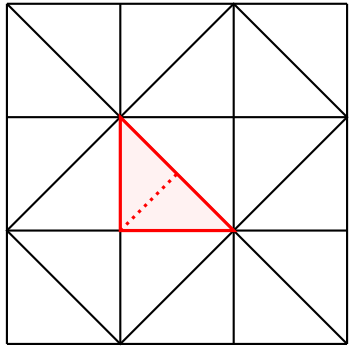

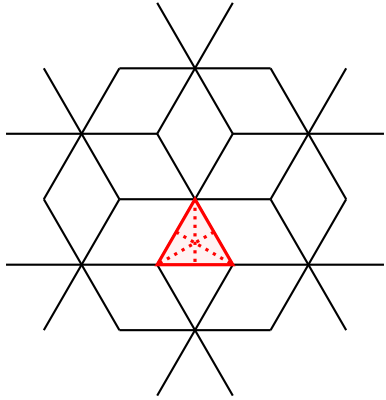
\begin{figure}[h]
    \centering
    \begin{tikzpicture}[scale=1]
        % Draw inner hexagon
	    \foreach \i in {0,...,2}{
	        \draw[thick] ({cos(60 * \i)}, {sin(60 * \i)}) -- ({cos(60 * (\i + 3))}, {sin(60 * (\i + 3))});
	    }
	            
	    % Loop over each outer portion of hexagon
	    \foreach \t in {0,...,5}{
	    
	        % Define offset via circle parameterization, distance = sqrt2/2
	        \pgfmathsetmacro{\xshift}{1.732 * cos(60 * \t + 30)}
	        \pgfmathsetmacro{\yshift}{1.732 * sin(60 * \t + 30)}
	    
	            % Draw outer hexagons
	            \begin{scope}[shift=({\xshift,\yshift})]
	                % hexagon border
	                \draw[thick] \foreach \i in {0,...,6}
	                { -- ({cos(60 * \i)},{sin(60 * \i)})};
	
	                % inner lines
	                \foreach \i in {0,...,2}{
	                \draw[thick] ({cos(60 * \i)}, {sin(60 * \i)}) -- ({cos(60 * (\i + 3))}, {sin(60 * (\i + 3))});
	                }
	            \end{scope}
	    }
	
	    % Draw highlighted region
	    \draw[red, very thick,fill=red!5] (0,0) -- ({cos(240)},{sin(240)}) -- ({cos(-60)},{sin(-60)}) -- cycle;
	    \draw[red, very thick, dotted] (-0.25, -0.43301) -- (0.5, -0.86603);
	    \draw[red, very thick, dotted] (0.25, -0.43301) -- (-0.5, -0.86603);
	    \draw[red, very thick, dotted] (0,0) -- (0, -0.86603);
    \end{tikzpicture}
    \caption{Showcasing equilateral triangle tiling and internal symmetries in red}
    \label{fig:lemma 2.1.2}
\end{figure}

\begin{figure}[h]
    \centering
    \begin{tikzpicture}[scale=0.5]
        % Draw square grid
        \foreach \i in {0,...,3}{
            \draw[thick] ({4 * \i}, 9) -- ({4 * \i}, 0);
            \draw[thick] (0, {3 * \i}) -- (12, {3 * \i});
        }

        % Draw highlighted region
        \draw[red, very thick,fill=red!5] (4,3) --  (8,3) -- (8,6) -- (4,6) -- cycle;
        \draw[red, very thick, dotted] (6,3) -- (6,6);
        \draw[red, very thick, dotted] (4,4.5) -- (8,4.5);
    \end{tikzpicture}
    \caption{Showcasing rectangle tiling and internal symmetries in red}
    \label{fig:lemma 2.1.3}
\end{figure}
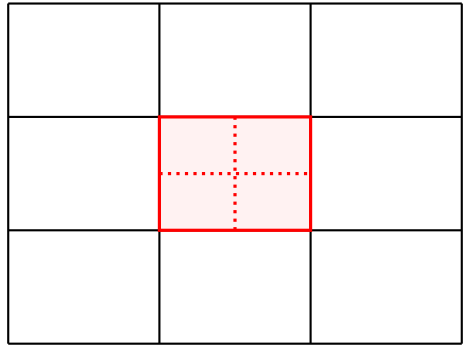

Now suppose $\Omega$ is a compact domain that tiles the plane evenly by reflections in its sides. Let $R$ denote the group generated by these reflections. We aim to show that $\Omega$ is a fundamental domain for the action of $R$ on $\mathbb{R}^2$.

By definition (see Definition~\ref{def: tile evenly}), the action of $R$ on $\Omega$ generates a tiling of the plane. That is, every point in $\mathbb{R}^2$ lies in the image of $\Omega$ under some element of $R$, so each orbit intersects $\Omega$ in at least one point.

Suppose, for contradiction, that $\Omega$ is not a fundamental domain for $R$. Then there exist distinct group elements $g_1, g_2 \in R$ and a point $p \in \Omega$ such that $g_1(p) = g_2(p)$, i.e., the orbit of $p$ intersects $\Omega$ more than once. Equivalently, some point of $\mathbb{R}^2$ lies in the image of two distinct tiles, contradicting the tiling being non-overlapping, unless $g := g_2^{-1}g_1$ is a non-identity group element that maps $\Omega$ onto itself. That is, $\Omega$ admits a nontrivial symmetry from within $R$.

We now analyze whether such an internal symmetry is possible. From Proposition~\ref{prop: only a few tilings}, the only possible compact domains that tile the plane evenly by reflection in their sides are:
rectangles, equilateral triangles, isosceles right triangles (45-45-90), and 30-60-90 triangles.

\textbf{Case 1:} If $\Omega$ is a 30-60-90 triangle, then it has no internal reflectional symmetry. Thus, no nontrivial element of $R$ maps $\Omega$ onto itself. Therefore, $\Omega$ must be a fundamental domain.

\textbf{Case 2:} If $\Omega$ is a rectangle, equilateral triangle, or isosceles right triangle, then it does admit internal symmetries (e.g., reflection along diagonals or axes). However, we claim that no such internal symmetry arises from the group $R$ unless it corresponds to a reflection in one of the sides of $\Omega$.

To see this, consider the tilings generated by each of these shapes, as shown in Figures 1–3. In each case, we examine a single tile and its immediate neighbors. Although these shapes may have internal symmetries (such as reflections across diagonals or midlines), these symmetries do not coincide with the sides used in the tiling process, nor are they preserved across adjacent tiles.

Since the tilings are periodic, any symmetry of a tile arising from the reflection group $R$ would appear consistently throughout the tiling. However, the figures clearly show that these internal symmetries are not replicated in other tiles, meaning they are not part of $R$. Therefore, no nontrivial element of $R$ maps $\Omega$ onto itself.

Therefore, in all cases, $\Omega$ does not overlap with its images under distinct elements of $R$, and hence $\Omega$ is a fundamental domain for the action of $R$.
\end{proof}

\section{Looping Billiard Trajectories} \label{sec: billiards}

Next, we extract information about the billiard trajectories that start and end at a point $x$ from the local Weyl counting function at $x$. We will use the wave equation to bridge these two concepts. We start with a review of some facts about the wave equation in the plane. For references, see Evans' standard text on PDEs \cite{evansPartialDifferentialEquations2022}.

\subsection{Some facts about the wave equation for planar domains}

Let $f$ be a $C^2$ function on $\R^2$. We wish to study the solution $u(t,x)$, as a function of $(t,x) \in \R \times \R^2$, to the wave equation
\begin{equation} \label{eq: wave equation}
    \frac{\partial^2 u}{\partial t^2} - \Delta u = 0 \qquad \text{ with } \qquad \begin{cases}
        u(0,x) = 0,  \\
        \frac{\partial u}{\partial t}(0,x) = f(x).
    \end{cases}
\end{equation}
The solution $u$ can be conveniently expressed as a convolution 
\begin{equation}\label{eq: convolution with fundamental solution}
    u(t,x) = E_t * f(x)
\end{equation}
where, for $t > 0$, $E_t$ is the fundamental solution
\begin{equation} \label{eq: fundamental solution}
    E_t(x) = \begin{cases}
        \frac 1 {2\pi} (t^2 - |x|^2)^{- \frac 1 2} & \text{ if $|x| \leq t$}, \\
        0 & \text{ otherwise.}
    \end{cases}
\end{equation}
Note, $E_t$ is an $L^p$ function for each $1 \leq p < 2$, and hence the convolution above is pointwise-defined provided $f \in L^q$ for some $2 < q \leq \infty$. From this, we can deduce a number of facts:
\begin{enumerate}
    \item (Linearity) The solution $u$ depends linearly on $f$.
    \item (Uniqueness)  If $u(t,x) = 0$ for all $t > 0$ and $x \in \R^2$, then $f = 0$.
    \item (Huygens' Principle) If $f = 0$ on a ball of radius $t > 0$ centered at $x$, then $u(t,x) = 0$.
\end{enumerate}

Now we consider a function $f$ on a compact domain $\Omega \subset \R^2$ with piecewise-smooth boundary $\partial \Omega$. The wave equation is phrased very similarly as in the whole plane, but with an added \emph{boundary condition}. That is, we say $u$ solves the wave equation on $\Omega$ with Dirichlet boundary conditions if:
\begin{equation} \label{eq: wave equation 2}
    \frac{\partial^2 u}{\partial t^2} - \Delta u = 0, \quad \begin{cases}
        u(0,x) = 0,  \\
        \frac{\partial u}{\partial t}(0,x) = f(x),
    \end{cases}
    \quad \text{ and } \quad u(t,x) = 0 \text{ if } x \in \partial \Omega.
\end{equation}
Unfortunately, it is harder to find a fundamental solution for the wave equation. However, linearity still holds, and other methods (e.g. energy methods) can be employed to prove uniqueness and Huygens' principle.

\subsection{Extracting billiard trajectories from the local counting function} 

In what follows, $\Omega$ will be a bounded open polygonal domain that tiles the plane evenly by reflections over its sides. Given an isometry $g$ of the plane, $\sgn(g)$ is equal to the determinant of its linear part, i.e. $\sgn(g) = 1$ if $g$ is the identity, a translation, or a rotation, and $\sgn(g) = -1$ if $g$ is a reflection. The next lemma is a Poisson summation formula of sorts.

\begin{lemma}\label{lem: wave identity}
    Let $\Omega$ and $R$ be as in Proposition \ref{prop: characterization of evenly tiling domains}, and let $\{e_j : j \in \N\}$ be a Hilbert basis of Laplace-eigenfunctions for $L^2(\Omega)$ with Dirichlet boundary conditions. Then,
    \[
        \sum_{g \in R} \sgn(g) E_t(g \cdot x - x) = \sum_{j \in \N} \frac{\sin(t \lambda_j)}{\lambda_j} |e_j(x)|^2
    \]
    for all $x$ in the interior of $\Omega$.
\end{lemma}

\begin{proof}
    Let $f$ be a smooth function supported in the interior of $\Omega$. Let $u$ be the solution to the wave equation \eqref{eq: wave equation}. Now let
    \[
        \tilde f(x) = \sum_{g \in R} \sgn(g) f(g \cdot x)
    \]
    and $\tilde u$ be the solution to the wave equation with initial data $\tilde f$. Since the wave equation is linear (or by \eqref{eq: convolution with fundamental solution}), we find
    \[
        \tilde u(t,x) = \sum_{g \in R} \sgn(g) u(t,g \cdot x).
    \]
    The right hand side expands to
    \begin{equation}\label{eq: rhs}
        \sum_{g \in R} \sgn(g) E_t * f(g \cdot x) = \int_{\R^2} \left( \sum_{g \in R} \sgn(g) E_t(g \cdot x - y) \right) f(y) \, dy.
    \end{equation}

    Now we evaluate the left hand side. Restrict $\tilde u$ to the fundamental domain $\Omega$. Note that $\tilde u$ satisfies the initial conditions
    \[
        \begin{cases}
            \tilde u(0,x) = 0 \\
            \frac{\partial \tilde u}{\partial t} (0,x) = f(x)
        \end{cases}
        \qquad \text{ for each } x \in \Omega.
    \]
    Recall, we have chosen $\Omega$ so that for each $x \in \Omega$, $x \in \partial \Omega$ if and only if there exists $g \in R$ with $\sgn(g) = -1$ and $g \cdot x = x$. Hence, if $x \in \partial \Omega$ with such a $g \in R$,
    \begin{align*}
        \tilde u(t,x) &= \sum_{h \in R} \sgn(h) u(t, h \cdot x) \\
        &= \sum_{h \in R} \sgn(hg) u(t, hg \cdot x) \\
        &= \sum_{h \in R} \sgn(h) \sgn(g) u(t, h \cdot x) \\
        &= - \sum_{h \in R} \sgn(h) u(t, h \cdot x) \\
        &= - \tilde u(t,x),
    \end{align*}
    and hence $\tilde u$ vanishes on $\partial \Omega$. Hence, $\tilde u$ is also subject to Dirichlet boundary conditions. Now observe
    \[
        \tilde v(t,x) = \sum_{j \in \N} \left(\int_{\Omega} f(y) \overline{e_j(y)} \, dy \right) \frac{\sin(t \lambda_j)}{\lambda_j} e_j(x)
    \]
    also satisfies the wave equation on $\Omega$, with the same initial data as $\tilde u$, and the same boundary conditions as $\tilde u$. By uniqueness, $\tilde u = \tilde v$, and hence
    \[
        \tilde u(t,x) = \int_{\R^2} \left( \sum_j \frac{\sin(t\lambda_j)}{\lambda_j} e_j(x) \overline{e_j(y)} \right) f(y) \, dy.
    \]

    To summarize, we have shown, for all smooth functions $f$ supported in the interior of $\Omega$, that
    \[
        \int_{\R^2} \left( \sum_j \frac{\sin(t\lambda_j)}{\lambda_j} e_j(x) \overline{e_j(y)} \right) f(y) \, dy = \int_{\R^2} \left( \sum_{g \in R} \sgn(g) E_t(g \cdot x - y) \right) f(y) \, dy.
    \]
    Replacing $f$ with an approximate identity centered at $x$ and taking the limit yields the desired identity.
\end{proof}

To make this useful, we must write the left side of the identity in Lemma \ref{lem: wave identity} in terms of billiard trajectories. In what follows, we let
\begin{equation}\label{eq: def looping spectrum}
    \ell(t,x) = \sum \sgn(g)
\end{equation}
where the sum is taken over $g \in R$ with $|g \cdot x - x| = t$. Appealing to the definition of $E_t$ in \eqref{eq: fundamental solution}, we have
\begin{equation}\label{eq: looping spectrum transform}
    \sum_{g \in R} \sgn(g) E_t(g \cdot x - x) = \frac{1}{2\pi} \sum_{0 \leq s < t} \ell(s,x) (t^2 - s^2)^{-\frac 1 2},
\end{equation}
where the latter sum is taken over the finitely many values of $s$ for which $\ell(s,x) \neq 0$. In fact, the right side is uniquely determined by $\ell(s,x)$.

To motivate the upcoming Lemma \ref{lem: billiards}, we make an observation regarding the structure of the right-hand side sum from \eqref{eq: looping spectrum transform}:
\begin{lemma}\label{lem: unique transform}
    Let $s_n$ be a strictly increasing sequence of nonnegative real numbers with $s_n \to \infty$, and let $a_n$ be a sequence of integers. Suppose
    \[
        \sum_{s_n < t} a_n (t^2 - s_n^2)^{-\frac 1 2} = 0 \qquad \text{ for each } t > 0.
    \]
    Then, $a_n = 0$ identically for each $n$.
\end{lemma}

\begin{proof}
    We show $a_n = 0$ for every $n$ by induction, the inductive step assuming $a_1 = \cdots = a_{n-1} = 0$ (a vacuous hypothesis when $n = 1$). Choose any $t$ with $s_n < t < s_{n+1}$. Every term of index $m < n$ vanishes by the inductive hypothesis. Hence
    \[
        0 = \sum_{s_m < t} a_m (t^2 - s_m^2)^{-\frac 1 2} = a_n (t^2 - s_n^2)^{-\frac 1 2}.
    \]
    Since $(t^2 - s_n^2)^{-\frac 1 2} \neq 0$, we conclude $a_n = 0$, completing the induction.
\end{proof}

Combining Lemmas \ref{lem: wave identity} and \ref{lem: unique transform} yields the key lemma of this section:

\begin{lemma}\label{lem: billiards}
    Suppose $x$ and $y$ are points in the interior of $\Omega$ such that $\ell(t,x) \neq \ell(t,y)$ for some $t > 0$. Then, $N(x,\lambda) \neq N(y,\lambda)$ for some $\lambda > 0$.
\end{lemma}

\begin{proof}
    We prove the contrapositive. Suppose $N(x,\lambda) = N(y,\lambda)$ for all $\lambda$. Then,
    \[
        \sum_{\lambda_j = \lambda} |e_j(x)|^2 = \sum_{\lambda_j = \lambda} |e_j(y)|^2 \qquad \text{ for each $\lambda$}.
    \]
    It follows
    \begin{align*}
        \sum_j \frac{\sin(t\lambda_j)}{\lambda_j} |e_j(x)|^2 &= \sum_\lambda \frac{\sin(t\lambda)}\lambda \sum_{\lambda_j = \lambda}|e_j(x)|^2 \\
        &= \sum_\lambda \frac{\sin(t\lambda)}\lambda \sum_{\lambda_j = \lambda}|e_j(y)|^2 \\
        &= \sum_j \frac{\sin(t\lambda_j)}{\lambda_j} |e_j(y)|^2.
    \end{align*}
    Lemma \ref{lem: wave identity} and \eqref{eq: looping spectrum transform} ensure
    \[
        0 = \frac 1 {2\pi} \sum_{0 \leq s < t} (\ell(s,x) - \ell(s,y)) (t^2 - s^2)^{-\frac 1 2} \qquad \text{ for all $t > 0$.}
    \]
    Lemma \ref{lem: unique transform} ensures $\ell(t,x) = \ell(t,y)$ for each $t \geq 0$.
\end{proof}

%%%%%%%%%%%%%%%%%%%%%%%%%%%%%%
% CALCULATING EIGENFUNCTIONS %
%%%%%%%%%%%%%%%%%%%%%%%%%%%%%%

\section{Fourier series for evenly-tiling domains} \label{sec: Fourier series}

Recall the \emph{orthogonal group} $O(\R^2)$ is the group of $2 \times 2$ matrices $U$ with real entries for which $U^t U = I$. It is the group of linear isometries on $\R^2$.

Let $\Omega$ be a triangle that tiles the plane evenly. By Proposition \ref{prop: only a few tilings}, $\Omega$ may be one of three types of triangle --- the equilateral, isosceles right, and 30-60-90 triangle. Without loss of generality, we will assume each is in standard position, as depicted by the shaded triangles in Figure \ref{fig: triangles in standard position}.

\begin{figure}[ht]
  \centering
  % (1) 45-45-90 -> square of 8 triangles (hypotenuse = 1)
  \begin{tikzpicture}[scale=1.7, baseline={(0,0)}]
    \useasboundingbox (-1.2,-1.25) rectangle (1.2,1.25);
    \fill[gray!30] (0,0) -- (0:{1/sqrt(2)}) -- (45:1) -- cycle;
    \draw[->,gray] (-1.2,0) -- (1.2,0);
    \draw[->,gray] (0,-1.25) -- (0,1.25);
    \draw[thick] (45:1) -- (135:1) -- (225:1) -- (315:1) -- cycle;
    \draw[thick] (225:1) -- (45:1);
    \draw[thick] (315:1) -- (135:1);
    \draw[thick] (180:{1/sqrt(2)}) -- (0:{1/sqrt(2)});
    \draw[thick] (270:{1/sqrt(2)}) -- (90:{1/sqrt(2)});
  \end{tikzpicture}\hfill
  % (2) equilateral -> hexagon of 6 triangles
  \begin{tikzpicture}[scale=1.7, baseline={(0,0)}]
    \useasboundingbox (-1.2,-1.25) rectangle (1.2,1.25);
    \fill[gray!30] (0,0) -- (30:1) -- (330:1) -- cycle;
    \draw[->,gray] (-1.2,0) -- (1.2,0);
    \draw[->,gray] (0,-1.25) -- (0,1.25);
    \draw[thick] (30:1) -- (90:1) -- (150:1) -- (210:1) -- (270:1) -- (330:1) -- cycle;
    \foreach \a in {30,90,150,210,270,330} { \draw[thick] (0,0) -- (\a:1); }
  \end{tikzpicture}\hfill
  % (3) 30-60-90 -> same hexagon split into 12 triangles
  \begin{tikzpicture}[scale=1.7, baseline={(0,0)}]
    \useasboundingbox (-1.2,-1.25) rectangle (1.2,1.25);
    \fill[gray!30] (0,0) -- (0:{sqrt(3)/2}) -- (30:1) -- cycle;
    \draw[->,gray] (-1.2,0) -- (1.2,0);
    \draw[->,gray] (0,-1.25) -- (0,1.25);
    \draw[thick] (30:1) -- (90:1) -- (150:1) -- (210:1) -- (270:1) -- (330:1) -- cycle;
    \foreach \a in {30,90,150,210,270,330} { \draw[thick] (0,0) -- (\a:1); }
    \foreach \a in {0,60,120,180,240,300} { \draw[thick] (0,0) -- (\a:{sqrt(3)/2}); }
  \end{tikzpicture}
  \caption{In each of the three diagrams, the shaded triangle depicts $\Omega$ in standard position, along with its (unshaded) reflections across sides incident to the origin. The union of $\Omega$ with its reflections is a fundamental domain (in fact a Voronoi domain) of the lattice of translations $\Lambda$ in Proposition \ref{prop: structure of reflection group}, while $G$ is a symmetry group of the fundamental domain.}
  \label{fig: triangles in standard position}
\end{figure}
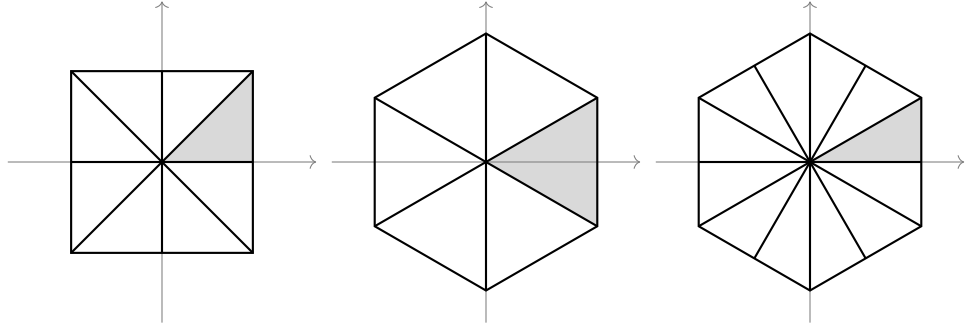

\begin{proposition}\label{prop: structure of reflection group}
    Let $\Omega$ be any of the three triangles depicted in Figure \ref{fig: triangles in standard position}, and let $R$ be the group generated by the reflections over its sides. Then, $R$ is the inner semi-direct product
    \[
        R = \Lambda G
    \]
    where $\Lambda$ is a lattice of translations in $\R^2$ and $G$ is a finite subgroup of the orthogonal group $O(\R^2)$. For each of the three cases, let $L$ be the length of the longest side of the triangle $\Omega$. We have:
    \begin{enumerate}
        \item (45-45-90 triangle) $\Lambda$ is the group of translations associated with the lattice $\sqrt 2 L \Z^2$, and $G$ is the dihedral group $G = \langle \sigma, \tau \rangle$ of order $8$ where
        \[
            \sigma = \begin{bmatrix}
                0 & -1 \\
                1 & 0
            \end{bmatrix}
            \qquad \text{ and } \qquad
            \tau = \begin{bmatrix}
                1 & 0 \\
                0 & -1
            \end{bmatrix}.
        \]
        \item (Equilateral triangle) $\Lambda$ is the group of translations associated with the lattice
        \[
            L\begin{bmatrix}
                \frac {\sqrt 3} 2 & \frac{\sqrt 3} 2 \\
                \frac 3 2 & -\frac 3 2
            \end{bmatrix}
            \Z^2,
        \]
        and $G$ is the dihedral group $G = \langle \sigma, \tau \rangle$ of order $6$ where
        \[
            \sigma = \begin{bmatrix}
                -\frac 1 2 & -\frac {\sqrt 3} 2 \\
                \frac {\sqrt 3} 2 & -\frac 1 2
            \end{bmatrix}
            \qquad \text{ and } \qquad
            \tau = \begin{bmatrix}
                -1 & 0 \\
                0 & 1
            \end{bmatrix}.
        \]
        \item (30-60-90 triangle) $\Lambda$ is the same lattice of translations as for the equilateral triangle. $G$ is the dihedral group $G = \langle \sigma, \tau \rangle$ of order $12$ where
        \[
            \sigma = \begin{bmatrix}
                \frac 1 2 & -\frac {\sqrt 3} 2 \\
                \frac {\sqrt 3} 2 & \frac 1 2
            \end{bmatrix}
            \qquad \text{ and } \qquad
            \tau = \begin{bmatrix}
                1 & 0 \\
                0 & -1
            \end{bmatrix}.
        \]
    \end{enumerate}
\end{proposition}

We assume $\Omega$ is in standard position so that $G$ will be an honest subgroup of $R$; otherwise $G$ need not be contained in $R$.

\begin{proof}
    For each $x \in \R^2$, let $\tau_x : \R^2 \to \R^2$ denote the translation $\tau_x(y) = x + y$. Now consider the map $\phi : R \to \Isom(\R^2)$ given by
    \[
        \phi(\rho) = \tau_{-\rho(0)} \rho
    \]
    for each $\rho \in R$. One quickly verifies $\phi$ is a group homomorphism and that $\phi(\rho)(0) = \tau_{-\rho(0)}\rho(0) = 0$, and hence $\phi(\rho)$ is a linear isometry (rather than an affine isometry more broadly), and hence belongs to the orthogonal group $O(\R^2)$. 
    
    Let $\Lambda$ and $G$ denote the kernel and image of $\phi$, respectively. For each $\rho \in R$, we write
    \[
        \rho = \tau_{\rho(0)} \phi(\rho),
    \]
    note that $\phi(\rho) \in G$ and that $\phi(\tau_{\rho(0)}) = \tau_{-\tau_{\rho(0)}(0)} = I$, and hence $\tau_{\rho(0)} \in \Lambda$. Hence, we will be done proving the general claim in the proposition provided we can prove (i) $G$ is a finite subgroup of $R$, (ii) $\Lambda$ and $G$ have trivial intersection, and (iii) $\Lambda$ is a lattice of translations. For the specific claims of (1), (2), and (3), we refer to Figure \ref{fig: triangles in standard position} as a suitably convincing picture, and otherwise leave the details to the reader.

    Let $\rho_1, \rho_2$, and $\rho_3$ denote the reflections over the three sides of $\Omega$ in standard position, where here we let $\rho_1$ be the reflection over the side extending into the first quadrant, and $\rho_3$ the reflection over the side opposite the origin. We claim that $\phi(\rho_3) \in \langle \rho_1, \rho_2 \rangle$, so that
    \[
        G = \langle \phi(\rho_1), \phi(\rho_2), \phi(\rho_3) \rangle = \langle \rho_1, \rho_2 \rangle.
    \]
    We leave it to the reader to verify, in each of the three cases (1), (2), and (3) that $G$ is the indicated dihedral group. We note
    \[
        \phi(\rho_3) = \tau_{-\rho_3(0)} \rho_3.
    \]
    One quickly verifies that $\phi(\rho_3)(x,y) = (-x,y)$, which we can write in each of three cases as
    \begin{align*}
        \phi(\rho_3) &= \rho_1 \rho_2 \rho_1 && \text{(45-45-90 triangle)} \\
        \phi(\rho_3) &= \rho_1 \rho_2 \rho_1 && \text{(equilateral triangle)} \\
        \phi(\rho_3) &= \rho_1 \rho_2 \rho_1 \rho_2 \rho_1 && \text{(30-60-90 triangle)}.
    \end{align*}
    We conclude that $G$ is a finite subgroup of $R \cap O(\R^2)$, and we have proved claim (i).

    Next, we claim that $\Lambda$ is a translation subgroup of $\R^2$. To see it, we note that if $\rho \in \Lambda = \ker \phi$, then $I = \tau_{-\rho(0)} \rho$ and hence $\rho = \tau_{\rho(0)}$. Claim (ii) now quickly follows. If $\rho \in \Lambda \cap G$, then $\rho$ is a translation which fixes the origin, and hence is the identity. We also take this opportunity to justify why $R = \Lambda G$.

    Now we finish by proving (iii). By \cite[\S1.3 Theorem 2]{gruberGeometryNumbers1987}, to prove $\Lambda$ is a lattice, it suffices to show it is discrete and not contained in a one-dimensional subspace of $\R^2$. Again, Figure \ref{fig: triangles in standard position} offers a convincing picture, but we elect for a general argument here.
    To show $\Lambda$ is discrete, consider any open disk $D$ in $\Omega$ of radius $\epsilon > 0$. Now, consider the translation $\tau_x \in \Lambda$ and suppose $|x| < \epsilon$. Then, if $x_0$ denotes the center of $D$, we have
    \[
        |\tau_x(x_0) - x_0| = |x| < \epsilon,
    \]
    and hence $\tau_x(D)$ and $D$ have nonempty and, since both $\tau_x(D)$ and $D$ are open disks, positive measure intersection. Since $D$ is contained in a fundamental domain $\Omega$, we conclude $\tau_x = I$, and hence $\Lambda$ is discrete.
    To show $\Lambda$ is not contained in a one-dimensional subspace, first pick representatives $\rho_1, \ldots, \rho_n$ from the $n = \# G$ cosets of $\Lambda$ in $R$, and note
    \[
        \tilde \Omega := \bigcup_{j = 1}^n \rho_j(\Omega)
    \]
    is a fundamental domain for $\Lambda$. Let $D_r$ be an open disk of radius $r$ large enough to contain $\tilde \Omega$ and note that
    \[
        \R^2 = \bigcup_{\tau_x \in \Lambda} (D_r + x).
    \]
    We then quickly conclude that $\{x \in \R^2 : \tau_x \in \Lambda \}$ is not contained in a line, since otherwise $D_r + x$ for $\tau_x \in \Lambda$ would not cover $\R^2$.
\end{proof}

Now that we understand the structure of the tiling group $R$, we may describe a convenient basis for $L^2(\Omega)$ constructed from the Fourier basis for $L^2(\R^2/\Lambda)$. From this, we can explicitly calculate as many eigenfunctions as we need.

Let $\Omega$, $R$, $\Lambda$, and $G$ be as in Proposition \ref{prop: structure of reflection group}, and let $L^2_{\text{odd}}(\R^2/R)$ denote the space of functions on $\R^2$ satisfying the modified periodicity condition
\[
    f \circ \rho = \sgn(\rho) f
\]
where $\sgn(\rho)$ is the determinant, either $1$ or $-1$, of the linear part of $\rho$ as an affine transformation, and where the restriction of $f$ to $\Omega$ is in $L^2(\Omega)$. Note that $L^2_\text{odd}(\R^2/R) \subset L^2(\R^2/\Lambda)$, since $\Lambda$ is a subgroup of $R$ and $\sgn(\tau) = 1$ for each $\tau \in \Lambda$. Let $\tilde \Omega$ be the square or hexagonal fundamental domain of $\Lambda$ as depicted in Figure \ref{fig: triangles in standard position}, and let $f \in L^2_\text{odd}(\R^2/R)$. Then, by Corollary \ref{cor: lattice fourier series}, we may expand $f$ in terms of its Fourier series
\begin{align*}
    f(x) &= \sum_{\xi \in \widehat \Lambda} e^{2\pi i x \cdot \xi} \frac 1 {|\tilde \Omega|} \int_{\tilde \Omega} e^{-2\pi i y \cdot \xi} f(y) \, dy 
    \\
    &= \sum_{\xi \in \widehat \Lambda} e^{2\pi i x \cdot \xi} \frac 1 {|\tilde \Omega|} \int_{\tilde \Omega} e^{-2\pi i y \cdot \xi} \frac{1}{\#G} \sum_{g \in G} \sgn(g) f(gy) \, dy 
    \\
    &= \sum_{\xi \in \widehat \Lambda} e^{2\pi i x \cdot \xi} \frac 1 {\#G} \sum_{g \in G} \sgn(g) \frac 1 {|\tilde \Omega|} \int_{\tilde \Omega} e^{-2\pi i g^{-1} y \cdot \xi} f(y) \, dy 
    \\
    &= \sum_{\xi \in \widehat \Lambda} e^{2\pi i x \cdot \xi} \frac 1 {|\tilde \Omega|} \int_{\tilde \Omega} \left(\frac 1 {\#G} \sum_{g \in G} \sgn(g) e^{-2\pi i y \cdot g\xi}\right) f(y) \, dy 
    \\
    &= \sum_{\xi \in \widehat \Lambda} e^{2\pi i x \cdot \xi} \frac 1 {|\Omega|} \int_{\Omega} \left(\frac 1 {\#G} \sum_{g \in G} \sgn(g) e^{-2\pi i y \cdot g\xi}\right) f(y) \, dy,
\end{align*}
The third line follows from a change of variables $y \mapsto g^{-1} y$, where we note $g\tilde \Omega = \tilde \Omega$ for each $g \in G$. The fourth line follows since $g$ is an orthogonal transformation on $\R^2$. The fifth line follows since the integrand is now invariant under change of variables by an element in $G$. In fact, we now note that the coefficient
\[
     a_\xi := \frac 1 {|\Omega|} \int_{\Omega} \left(\frac 1 {\#G} \sum_{g \in G} \sgn(g) e^{-2\pi i y \cdot g\xi}\right) f(y) \, dy
\]
satisfies
\[
    a_{g^{-1}\xi} = \sgn(g) a_\xi \qquad \text{ for each } g \in G,
\]
and hence
\begin{align*}
    f(x) &= \sum_{\xi \in \widehat \Lambda} e^{2\pi i x \cdot \xi} a_\xi \\
    &= \sum_{\xi \in \widehat \Lambda} e^{2\pi i x \cdot \xi} \frac 1 {\#G} \sum_{g \in G} \sgn(g) a_{g^{-1}\xi} \\
    &= \frac 1 {\#G} \sum_{g \in G} \sgn(g) \sum_{\xi \in \widehat \Lambda} e^{2\pi i x \cdot \xi} a_{g^{-1}\xi} \\
    &= \frac 1 {\#G} \sum_{g \in G} \sgn(g) \sum_{\xi \in \widehat \Lambda} e^{2\pi i x \cdot g\xi} a_{\xi} \\
    &= \sum_{\xi \in \widehat \Lambda} \left( \frac 1 {\# G} \sum_{g \in G} \sgn(g) e^{2\pi i x \cdot g\xi} \right) a_\xi.
\end{align*}
To put everything together:

\begin{lemma}\label{lem: automorphic basis}
    Let $\Omega$, $R$, $\Lambda$, and $G$ be as in Proposition \ref{prop: structure of reflection group}. Then for each $f \in L^2(\Omega)$, we have
    \[
        f(x) = \sum_{\xi \in \widehat \Lambda} \varphi_\xi(x) \int_{\Omega} f(y) \overline{\varphi_\xi(y)} \, dy
    \]
    where
    \[
        \varphi_\xi(x) = \frac 1 {\# G} \sum_{g \in G} \sgn(g) e^{2\pi i x \cdot g\xi}
    \]
    and where $\varphi_\xi(\rho x) = \sgn(\rho) \varphi_\xi(x)$ for each $x \in \R^2$ and $\rho \in R$. In particular, $\varphi_\xi$ vanishes on the boundary of $\Omega$ and
    \[
        \Delta \varphi_\xi = -|2\pi \xi|^2 \varphi_\xi.
    \]
\end{lemma}

\begin{proof}
    We have already proved the first identity, and what remains is to prove the $R$-invariance of $\varphi_\xi$. Fix $\rho \in R$ and write, per Proposition \ref{prop: structure of reflection group},
    \[
        \rho = \tau h \qquad \text{ for some $\tau \in \Lambda$, $h \in G$.}
    \]
    We write
    \begin{align*}
        \varphi_\xi(\rho x) &= \varphi_\xi(\tau h x) \\
        &= \frac 1 {\#G} \sum_{g \in G}  \sgn(g) e^{2\pi i \tau h x \cdot g \xi} \\
        &= \frac 1 {\#G} \sum_{g \in G} \sgn(g) e^{2\pi i h x \cdot g \xi} && \text{(the exponential is $\Lambda$-periodic)} \\
        &= \sgn(h) \frac 1 {\#G} \sum_{g \in G} \sgn(h^{-1} g) e^{2\pi i x \cdot h^{-1} g \xi} && \text{($h \in O(\R^2)$)} \\
        &= \sgn(h) \frac 1 {\#G} \sum_{g \in G} \sgn(g) e^{2\pi i x \cdot g \xi} \\
        &= \sgn(\rho) \varphi_\xi(x).
    \end{align*}
    That $\varphi_\xi$ vanishes along the boundary of $\Omega$ follows immediately. To see that $\varphi_\xi$ is a Laplace eigenfunction, we first note that since $G \subset O(\R^2)$, $|g\xi| = |\xi|$ for each $g \in G$. Hence,
    \begin{multline*}
        \Delta \varphi_\xi(x) = \frac{1}{\# G} \sum_{g \in G} \sgn(g) \Delta e^{2\pi i x \cdot g\xi} = \frac{1}{\# G} \sum_{g \in G} -\sgn(g) |2\pi g\xi|^2 e^{2\pi i x \cdot g\xi} \\
        = \frac{1}{\# G} \sum_{g \in G} -\sgn(g) |2\pi \xi|^2 e^{2\pi i x \cdot g\xi} = -|2\pi \xi|^2 \varphi_\xi(x).
    \end{multline*}
\end{proof}

We are nearly ready to write down the Laplace eigenbasis in terms of Fourier exponentials, $\widehat \Lambda$ and $G$. Before we do, we introduce just a little more notation. Fix $\lambda \geq 0$ and consider the span
\[
    E_\lambda = \operatorname{span} \{\varphi_\xi : \xi \in \widehat \Lambda, \ |2\pi \xi| = \lambda \}.
\]
By the lemma above, $E_\lambda$ is contained in the set of Laplace eigenfunctions on $\Omega$ with Dirichlet boundary conditions. Hence, the spaces $E_\lambda$ and $E_\mu$ for $\lambda \neq \mu$ are orthogonal in $L^2(\Omega)$ (though this can be determined by direct computation). Select an orthonormal basis for each $E_\lambda$ and collect them into a set $\{e_j : j \in \N\}$. By construction, $\{e_j : j \in \N\}$ is an orthonormal system of Laplace eigenfunctions on $\Omega$ with Dirichlet boundary conditions, and by the lemma, spans $L^2(\Omega)$. We now give an explicit form to the eigenfunctions $e_j$.

\begin{lemma}\label{lem: orthonormal automorphic basis}
    Let $\Omega$, $R$, $\Lambda$, and $G$ be as in Proposition \ref{prop: structure of reflection group}, and let $E_\lambda$ be as above. Then, $E_\lambda$ admits an orthogonal basis of elements
    \[
        \sqrt{ \frac {\#G} {|\Omega|}} \varphi_\xi(x) = \frac 1 {\sqrt{\#G|\Omega|}} \sum_{g \in G} \sgn(g) e^{2\pi i x \cdot g\xi},
    \]
    one for each orbit $G\xi$, $\xi \in \widehat \Lambda$ and $|2\pi \xi| = \lambda$, for which the action of $G$ on $G\xi$ is free (i.e. if $g \in G$ fixes any point in $G\xi$, then $g$ is the identity).
\end{lemma}

\begin{proof}  
    First assume $G$ does not act freely on some orbit $G\xi$, so that there exists $g \in G \setminus \{I\}$ and $\eta \in G\xi$ for which $g\eta = \eta$. We claim there exists an element $h \in G$ with $\sgn(h) = -1$ that fixes $\xi$. If $\xi = 0$, the claim holds trivially, so assume $\xi \neq 0$, and hence $\eta \neq 0$. Let $k \in G$ be such that $k \xi = \eta$. Then,
    \[
        \xi = k^{-1}\eta = k^{-1}g\eta = k^{-1} g k \xi,
    \]
    so we may take $h = k^{-1} g k$. If $\sgn(h) = 1$, then $h$ is a rotation about the origin and fixes $\xi$ if and only if $h = I$, i.e. $g = I$, which is prohibited by assumption. Hence, $\sgn(h) = -1$.
    
    From the claim, it follows that when $G$ does not act freely on $G\xi$, $\varphi_\xi = 0$. Hence,
    \[
        E_\lambda = \operatorname{span} \{\varphi_\xi : \xi \in \widehat \Lambda, \ |2\pi \xi| = \lambda, \ \text{$G$ acts freely on $G\xi$}\}.
    \]
    Now suppose $\xi,\eta \in \widehat \Lambda$ with $|2\pi \xi| = |2\pi \eta| = \lambda$ and that $G$ acts freely on the orbits of $\xi$ and $\eta$. If $\xi$ and $\eta$ belong to the same orbit,
    \[
        \varphi_\xi = \pm \varphi_\eta,
    \]
    and if $\xi$ and $\eta$ belong to distinct orbits,
    \begin{align*}
        \int_\Omega \varphi_\xi(x) \overline{\varphi_\eta(x)} \, dx &= \frac{1}{(\# G)^2} \int_\Omega \sum_{g,h \in G} \sgn(g) \sgn(h) e^{2\pi i x \cdot (g \xi - h \eta)} \, dx \\
        &= \frac{1}{\# G} \int_{\tilde \Omega} \sum_{g,h \in G} \sgn(g) \sgn(h) e^{2\pi i x \cdot (g \xi - h \eta)} \, dx \\
        &= \frac{1}{\# G} \sum_{g,h \in G} \sgn(g) \sgn(h) \int_{\tilde \Omega} e^{2\pi i x \cdot (g \xi - h \eta)} \, dx \\
        &= \frac {|\tilde \Omega|} {\# G} \sum_{g,h \in G} \sgn(g) \sgn(h) \1_{g\xi = h\eta} \\
        &= \frac {|\tilde \Omega|} {\# G} \sum_{g,h \in G} \sgn(g) \sgn(gh) \1_{g\xi = gh\eta} \\
        &= |\tilde \Omega| \sum_{h \in G} \sgn(h) \1_{\xi = h\eta} \\
        &= \begin{cases}
            \pm \# G |\Omega|, & G\xi = G\eta, \\
            0 & G\xi \neq G\eta,
        \end{cases}
    \end{align*}
    where $\tilde \Omega$ is the square or hexagonal fundamental domain of the lattice $\Lambda$ depicted in Figure \ref{fig: triangles in standard position}, and where $\pm$ in the last line matches $\sgn(h)$ for which $\xi = h\eta$. The lemma follows.
\end{proof}

The lemma shows, directly, that $E_\lambda$ is finite-dimensional, and that, since $\widehat \Lambda$ is countable, there are countably many nontrivial spaces $E_\lambda$. We also note that each $E_\lambda$ consists of Laplace eigenfunctions of eigenvalue $-\lambda^2$ and satisfying Dirichlet boundary conditions. We let $\{e_j : j \in \N\}$ be the concatenation of orthonormalized basis elements for these spaces as in Lemma \ref{lem: orthonormal automorphic basis}, with the convention $e_j \in E_{\lambda_j}$. Lemma \ref{lem: automorphic basis} ensures $\{e_j : j \in \N\}$ is an orthonormal basis for $L^2(\Omega)$. Since we have relatively explicit forms for the eigenbasis elements $e_j$, we may calculate and write down the first few terms of the local Weyl counting function \eqref{def: local Weyl counting function},
\[
    N(x,\lambda) = \sum_{\lambda_j \leq \lambda} |e_j(x)|^2.
\]
We do this now for the 45-45-90 and 30-60-90 triangles.

\begin{lemma}\label{lem: 45-45-90 eigenfunction}
    The Laplace eigenfunctions on the 45-45-90 triangle $\Omega$ in standard position with Dirichlet boundary conditions are given by
    \begin{equation}\label{eq: 45-45-90 eigenfunction}
        \varphi_\xi(x_1,x_2) = \frac{1}{2}\left[\sin(2\pi n_2 x_1)\sin(2\pi n_1 x_2) - \sin(2\pi n_1 x_1)\sin(2\pi n_2 x_2)\right],
    \end{equation}
    indexed by $(n_1,n_2)\in\Z^2$, with eigenvalue $-\lambda_\xi^2$ where $\lambda_\xi^2 = 4\pi^2(n_1^2+n_2^2)$.
\end{lemma}

\begin{proof}
By Lemma~\ref{lem: automorphic basis}, the eigenfunction associated to $\xi \in \widehat{\Lambda}$ is
\begin{equation}\label{eq: 45-45-90 varphi def}
    \varphi_\xi(x) = \frac{1}{8}\sum_{g\in G}\sgn(g)\,e^{2\pi i x\cdot g\xi}.
\end{equation}
The lattice is $\Lambda = \Z^2$, so $A = I$ and $\widehat{\Lambda} = \Z^2$. A general dual lattice element is $\xi = (n_1, n_2)$ with $n_1, n_2 \in \Z$.

Recall from Proposition~\ref{prop: structure of reflection group}(1) that $\sigma$ and $\tau$ act on $\R^2$ via the matrices
\[
    \sigma = \begin{bmatrix} 0 & -1 \\ 1 & 0 \end{bmatrix}, \qquad \tau = \begin{bmatrix} 1 & 0 \\ 0 & -1 \end{bmatrix}.
\]
Since $\sigma$ applies a rotation by $90^\circ$, we have $\sigma^2 = -I$, so $\sigma^{k+2}\xi = -\sigma^k\xi$. The four rotation images therefore pair antipodally:
\[
    \eta_0 = \xi = (n_1, n_2), \qquad \eta_1 = \sigma\xi = (-n_2, n_1),
\]
with $\sigma^2\xi = -\eta_0$ and $\sigma^3\xi = -\eta_1$.  The four reflection images $\{\sigma^k\tau\xi\}$ pair similarly, with
\[
    \tau\eta_0 = (n_1, -n_2), \qquad \tau\eta_1 = (-n_2, -n_1).
\]
We expand \eqref{eq: 45-45-90 varphi def} by separating the four rotations ($\sgn = +1$) from the four reflections ($\sgn = -1$):
\[
    8\,\varphi_\xi(x) = \sum_{k=0}^{3}e^{2\pi i x\cdot\sigma^k\xi} - \sum_{k=0}^{3}e^{2\pi i x\cdot\sigma^k\tau\xi}.
\]
Using the antipodal pairing $\sigma^{k+2}\xi = -\sigma^k\xi$, each sum collapses into two cosines:
\[
    8\,\varphi_\xi(x) = 2\sum_{k=0}^{1}\left[\cos(2\pi x\cdot\eta_k) - \cos(2\pi x\cdot\tau\eta_k)\right].
\]
We apply the nice sum-to-product identity
\[
\cos(u+v) - \cos(u-v) = -2\sin(u)\sin(v).
\]
For $k=0$, $\eta_0 = (n_1, n_2)$ and $\tau\eta_0 = (n_1, -n_2)$, so
\[
    \cos(2\pi(n_1 x_1 + n_2 x_2)) - \cos(2\pi(n_1 x_1 - n_2 x_2)) = -2\sin(2\pi n_1 x_1)\sin(2\pi n_2 x_2).
\]
For $k=1$, $\eta_1 = (-n_2, n_1)$ and $\tau\eta_1 = (-n_2, -n_1)$, so
\[
    \cos(2\pi(-n_2 x_1 + n_1 x_2)) - \cos(2\pi(-n_2 x_1 - n_1 x_2)) = 2\sin(2\pi n_2 x_1)\sin(2\pi n_1 x_2).
\]
Summing these two contributions yields
\[
    8\,\varphi_\xi(x) = 2\bigl[-2\sin(2\pi n_1 x_1)\sin(2\pi n_2 x_2) + 2\sin(2\pi n_2 x_1)\sin(2\pi n_1 x_2)\bigr],
\]
which gives \eqref{eq: 45-45-90 eigenfunction}. The eigenvalue follows from $\lambda_\xi^2 = |2\pi\xi|^2 = 4\pi^2(n_1^2+n_2^2)$.
\end{proof}

\begin{lemma}\label{lem: 30-60-90 eigenfunction}
    The Laplace eigenfunctions on the 30-60-90 triangle $\Omega$ in standard position with Dirichlet boundary conditions are given by
    \begin{align}\label{eq: 30-60-90 eigenfunction}
        \varphi_\xi(x_1,x_2) = -\frac{1}{3}\Biggl[
        &\sin\!\left(\frac{2\pi(n_1+n_2)\,x_1}{\sqrt{3}}\right)\sin\!\left(\frac{2\pi(n_1-n_2)\,x_2}{3}\right) \notag\\
        +\,&\sin\!\left(\frac{2\pi n_2\,x_1}{\sqrt{3}}\right)\sin\!\left(\frac{2\pi(2n_1+n_2)\,x_2}{3}\right) \notag\\
        -\,&\sin\!\left(\frac{2\pi n_1\,x_1}{\sqrt{3}}\right)\sin\!\left(\frac{2\pi(n_1+2n_2)\,x_2}{3}\right)
        \Biggr],
    \end{align}
    indexed by $(n_1,n_2)\in\Z^2$, with eigenvalue $-\lambda_\xi^2$ where
    \[
        \lambda_\xi^2 = \frac{16\pi^2}{9}(n_1^2+n_1n_2+n_2^2), \qquad \xi = \begin{bmatrix} \dfrac{1}{\sqrt{3}} & \dfrac{1}{\sqrt{3}} \\[6pt] \dfrac{1}{3} & -\dfrac{1}{3} \end{bmatrix} \begin{bmatrix}
            n_1 \\ n_2
        \end{bmatrix}
    \]
\end{lemma}

The results of this lemma follow from the work in \cite{doncheskiQuantumMechanicalAnalysis2002}, but we include an argument here for completeness.

\begin{proof}
By Lemma~\ref{lem: automorphic basis}, the eigenfunction associated to $\xi \in \widehat{\Lambda}$ is
\begin{equation}\label{eq: 30-60-90 varphi def}
    \varphi_\xi(x) = \frac{1}{12}\sum_{g\in G}\sgn(g)\,e^{2\pi i x\cdot g\xi}.
\end{equation}
First, we compute the dual lattice $\widehat{\Lambda} = A^{-t}\Z^2$. Since $\det A = -\frac{3\sqrt{3}}{2}$, we get
\[
    A^{-t} = \begin{bmatrix} \dfrac{1}{\sqrt{3}} & \dfrac{1}{\sqrt{3}} \\[6pt] \dfrac{1}{3} & -\dfrac{1}{3} \end{bmatrix},
\]
so a general element of $\widehat{\Lambda}$ is
\[
    \xi = A^{-t}\binom{n_1}{n_2} = \binom{(n_1+n_2)/\sqrt{3}}{(n_1-n_2)/3}, \qquad (n_1,n_2)\in\Z^2.
\]
Next, we determine the action of $G$ on the dual lattice. Recall from Proposition~\ref{prop: structure of reflection group}(3) that $\sigma$ and $\tau$ act on $\R^2$ via the matrices
\[
    \sigma = \begin{bmatrix} \frac{1}{2} & -\frac{\sqrt{3}}{2} \\ \frac{\sqrt{3}}{2} & \frac{1}{2} \end{bmatrix}, \qquad \tau = \begin{bmatrix} 1 & 0 \\ 0 & -1 \end{bmatrix}.
\]
Computing $\sigma\xi$ and $\sigma^2\xi$ directly, we get the following: 
\begin{equation}\label{eq: def eta_k}
    \sigma\xi = \left(\frac{n_2}{\sqrt{3}},\, \frac{2n_1+n_2}{3}\right), \qquad 
    \sigma^2\xi = \left(\frac{-n_1}{\sqrt{3}},\, \frac{n_1+2n_2}{3}\right).
\end{equation}
We have $\sigma^3 = -I$, so $\sigma^{k+3}\xi = -\sigma^k\xi$.  Setting $\eta_k = \sigma^k\xi$, the six rotation images pair antipodally as $\{\eta_k, -\eta_k\}$ for $k = 0,1,2$.  Since $\tau$ negates only the second component, the six reflection images $\{\sigma^k\tau\xi\}$ likewise pair antipodally as $\{\tau\eta_k, -\tau\eta_k\}$.

We expand \eqref{eq: 30-60-90 varphi def} by separating the six rotations from the six reflections:
\[
    12\,\varphi_\xi(x) = \sum_{k=0}^{5}e^{2\pi i x\cdot\sigma^k\xi} - \sum_{k=0}^{5}e^{2\pi i x\cdot\sigma^k\tau\xi}.
\]
Just like in the 45-45-90 case, we use the antipodal pairings to collapse the six-term sums into three-term sums of cosines:
\[
    12\,\varphi_\xi(x) = 2\sum_{k=0}^{2}\left[\cos(2\pi x\cdot\eta_k) - \cos(2\pi x\cdot\tau\eta_k)\right].
\]
We apply the sum-to-product identity $\cos(u+v) - \cos(u-v) = -2\sin(u)\sin(v)$ to each difference.  With $\eta_k = (\eta_k^{(1)}, \eta_k^{(2)})$ and $\tau\eta_k = (\eta_k^{(1)}, -\eta_k^{(2)})$, 
\[
    \cos(2\pi x\cdot\eta_k) - \cos(2\pi x\cdot\tau\eta_k) = -2\sin\!\left(2\pi\,\eta_k^{(1)}x_1\right)\sin\!\left(2\pi\,\eta_k^{(2)}x_2\right).
\]
Substituting back yields
\[
    12\,\varphi_\xi(x) = -4\sum_{k=0}^{2}\sin\!\left(2\pi\,\eta_k^{(1)}x_1\right)\sin\!\left(2\pi\,\eta_k^{(2)}x_2\right).
\]
Inserting the expressions for $\eta_0 = \xi$, $\eta_1 = \sigma \xi$, and $\eta_2 = \sigma^2 \xi$ from \eqref{eq: def eta_k} yields \eqref{eq: 30-60-90 eigenfunction}. The eigenvalue follows from $\Delta\varphi_\xi = -|2\pi\xi|^2\varphi_\xi$, namely
\[
    \lambda_\xi^2 = 4\pi^2\!\left(\frac{(n_1+n_2)^2}{3} + \frac{(n_1-n_2)^2}{9}\right) = \frac{16\pi^2}{9}(n_1^2+n_1n_2+n_2^2). 
    \qedhere
\]
\end{proof}

%%%%%%%%%%%%%%%%%%%%%%%%%
% Proof of main theorem %
%%%%%%%%%%%%%%%%%%%%%%%%%

\section{Proof of Theorem \ref{thm: main 1}}\label{sec: proof of main theorem}

The following lemma, along with Lemma \ref{lem: billiards}, ensures the shortest two altitudes measured from the sides of an equilateral triangle to a point in its interior are audible.

\begin{lemma}\label{lem: altitudes are audible}
    Let $\Omega$ be an equilateral triangle, let $x$ be a point in its interior, and let $a_1 \leq a_2 \leq a_3$ be the perpendicular distances from $x$ to the three sides of $\Omega$. With $\ell$ as in \eqref{eq: def looping spectrum}, we have
    \begin{equation}\label{eq: altitude looping spectrum}
        \ell(t,x) = -\,\#\{\, i \in \{1,2,3\} : 2a_i = t \,\}
        \qquad \text{ for every } 0 < t \leq 2a_2 .
    \end{equation}
    In particular, $\ell(\,\cdot\,,x)$ determines the two shortest distances $a_1$ and $a_2$:
    if $x'$ is another interior point of $\Omega$ with perpendicular distances
    $a_1' \leq a_2' \leq a_3'$ and $\ell(\,\cdot\,,x) = \ell(\,\cdot\,,x')$, then $a_1 = a_1'$
    and $a_2 = a_2'$.
\end{lemma}

\begin{proof}
Let $\Omega := \triangle ABC$ be our equilateral triangle and $P$ be a point in its interior. For a closed polygonal path
$P X_1 \cdots X_n P$, let
\[
    |P X_1 \cdots X_n P| = PX_1 + X_1X_2 + \cdots + X_{n-1}X_n + X_nP
\]
denote its total length. Let $D,E,F$ be the feet of the perpendiculars from $P$ to the
sides $\overline{AB}, \overline{BC}, \overline{CA}$, respectively. Relabelling the vertices, we may assume
\[
    PD \leq PE \leq PF, 
\]
giving us $a_1 = PD, a_2 = PE,$ and $a_3 = PF$ as in the lemma hypothesis.

By Proposition~\ref{prop: characterization of evenly tiling domains}, $\Omega$ is a
fundamental domain for $R$; hence the tiles $\{g\Omega : g \in R\}$ cover the plane with
pairwise disjoint interiors, and $g \mapsto g\Omega$ is a bijection from $R$ onto the set
of tiles. Since $P$ and $g\cdot P$ lie in the interiors of
$\Omega$ and $g\Omega$, the segment $\sigma_g = \overline{P\,(g\cdot P)}$ meets the edges of
the tiling in finitely many points; folding the successive tiles that $\sigma_g$ passes
through back onto $\Omega$ folds $\sigma_g$ into a closed polygonal path
\[
    \gamma_g = P X_1 X_2 \cdots X_n P,
\]
where $X_i \in \partial\Omega$, $n = \#\text{edges $\sigma_g$ passes through}$, and $|\gamma_g| = |g\cdot P - P|$.

\begin{figure}[h]
    \centering
    % Lemma 5.1 (altitudes are audible): a looping billiard trajectory that starts
    % and ends at P, shown two ways. LEFT: the straight segment from P to g.P in the
    % unfolded chain of reflected tiles T_0=Omega, T_1, T_2 (dashed creases are the
    % reflected edges). RIGHT: folding the tiles back onto Omega turns that straight
    % segment into the closed polygonal path P X_1 X_2 P, with X_1 in BC and X_2 in CA.
    \begin{tikzpicture}[scale=3.1, every node/.style={font=\footnotesize}]

      % ---------- LEFT: unfolded chain ----------
      \coordinate (A)  at (0,0);
      \coordinate (B)  at (1,0);
      \coordinate (C)  at (0.5,0.86603);
      \coordinate (A1) at (1.5,0.86603);
      \coordinate (B2) at (1,1.73205);

      % tiles
      \fill[gray!16] (A) -- (B) -- (C) -- cycle;   % T_0 = Omega
      \fill[gray!5]  (B) -- (C) -- (A1) -- cycle;  % T_1
      \fill[gray!5]  (C) -- (A1) -- (B2) -- cycle; % T_2

      % outer boundary (solid) and creases (dashed = reflected edges)
      \draw[thick] (A) -- (B) -- (A1) -- (B2) -- (C) -- cycle;
      \draw[gray, dashed] (B) -- (C);   % crease T_0|T_1
      \draw[gray, dashed] (C) -- (A1);  % crease T_1|T_2

      % unfolded trajectory: straight segment P -> g.P
      \coordinate (P)   at (0.42,0.20);
      \coordinate (gP)  at (1.11684,1.12976);
      \coordinate (X1u) at (0.68238,0.55012);
      \coordinate (X2u) at (0.91909,0.86603);
      \draw[red, thick] (P) -- (gP);
      \fill (P) circle (0.6pt);   \node[below] at (P) {$P$};
      \fill (gP) circle (0.6pt);  \node[right] at (gP) {$g\cdot P$};
      \fill[red] (X1u) circle (0.5pt); \node[below right=-1pt] at (X1u) {$X_1$};
      \fill[red] (X2u) circle (0.5pt); \node[above left=-1pt] at (X2u) {$X_2$};

      \node at (0.27,0.30) {$\Omega$};

      % ---------- fold arrow ----------
      \draw[-{Stealth}, thick] (1.7,0.55) -- node[above]{fold} (2.15,0.55);

      % ---------- RIGHT: folded path in Omega ----------
      \begin{scope}[shift={(2.3,0)}]
        \coordinate (A')  at (0,0);
        \coordinate (B')  at (1,0);
        \coordinate (C')  at (0.5,0.86603);
        \fill[gray!16] (A') -- (B') -- (C') -- cycle;
        \draw[thick] (A') -- (B') -- (C') -- cycle;

        \coordinate (Pp)  at (0.42,0.20);
        \coordinate (X1p) at (0.68238,0.55012); % on B'C'
        \coordinate (X2p) at (0.29040,0.50300); % on C'A'
        \draw[red, thick] (Pp) -- (X1p) -- (X2p) -- cycle;
        \fill (Pp) circle (0.6pt);  \node[below] at (Pp) {$P$};
        \fill[red] (X1p) circle (0.5pt); \node[right] at (X1p) {$X_1$};
        \fill[red] (X2p) circle (0.5pt); \node[left] at (X2p) {$X_2$};

        \node[below left]  at (A') {$A$};
        \node[below right] at (B') {$B$};
        \node[above]       at (C') {$C$};
        \node at (0.52,0.08) {$\Omega$};
      \end{scope}
    \end{tikzpicture}
    \caption{Unfolding a looping trajectory at $P$: the straight segment $\overline{P\,(g\cdot P)}$ through the reflected tiles (left) folds back onto $\Omega$ to give the closed polygonal path $\gamma_g = P X_1 X_2 P$ (right), with $X_1 \in \overline{BC}$ and $X_2 \in \overline{CA}$.}
    \label{fig: unfolding}
\end{figure}
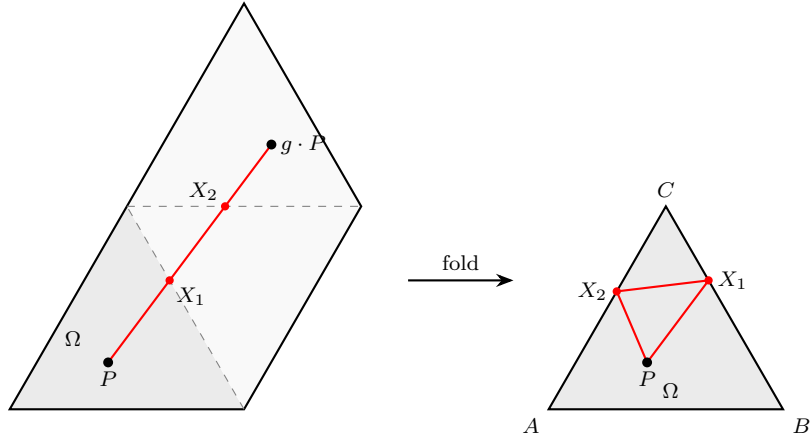

We call $g$ \emph{regular} if the open segment $(P,\,g\cdot P)$ avoids the vertices of the tiling, and \emph{singular} otherwise. 

For a regular $g$ we record two facts about $\gamma_g = P X_1 \cdots X_n P$. Firstly, consecutive vertices lie on distinct sides. Between two consecutive edge crossings the segment runs through the interior of a tile, so $\overline{X_i X_{i+1}}$ passes through $\interior \Omega$ and $X_i, X_{i+1}$ must lie on different sides of $\Omega$. Secondly, $n = 1$ if and only if $g \in \{\rho_1, \rho_2, \rho_3\}$, where $\rho_1, \rho_2, \rho_3$ are the reflections across the lines containing $\overline{AB}, \overline{BC}, \overline{CA}$, respectively. A single edge crossing means $g\Omega$ is a tile sharing the crossed side with $\Omega$, namely $\rho_i \Omega$ for $i=1, 2, 3$.

We see that $\gamma_{\rho_1}, \gamma_{\rho_2}, \gamma_{\rho_3}$ are just $PDP$, $PEP$, $PFP$ and have lengths $2a_1, 2a_2, 2a_3$, respectively. We are now ready to show that all other paths have length greater than $2a_2$. Fix $g \in R \setminus \{e\}$ with $n \ge 2$, and let $|g \cdot P - P| = t$. If $g$ is regular, then for some $k$, $X_k \in \overline{BC} \cup \overline{CA}$. If $g$ is singular, then for some $k$, $X_k \in \{A, B, C\}$. In either case, we choose the particular $k$. Splitting $\gamma_g$ at $X_k$ and
applying the triangle inequality to each half,
\[
    t = |\gamma_g|
      = \bigl(PX_1 + \cdots + X_{k-1}X_k\bigr) + \bigl(X_kX_{k+1} + \cdots + X_nP\bigr) \ge PX_k + X_kP \ge 2a_2 .
\]
The first inequality is strict for regular $g$ and the second inequality is strict for singular $g$, so $t > 2a_2$. Thus, $a_1$ and $a_2$ are always audible.
\end{proof}

\begin{proposition} \label{prop: equilateral triangle}
    One can always hear where an equilateral triangle is struck, up to symmetry.
\end{proposition}

\begin{proof}
We fix an equilateral triangle $T$ and two distinct points $P, Q$ in its interior with $N(P,\lambda) = N(Q,\lambda)$ for all $\lambda \geq 0$. By Lemma~\ref{lem: billiards}, $\ell(t, P) = \ell(t, Q)$ for all $t > 0$, where $\ell$ as in \eqref{eq: def looping spectrum}. Let $p_1 \le p_2 \le p_3$ be the distances from $P$ to the three sides of $T$, and let $q_1 \le q_2 \le q_3$ be those from $Q$. By Lemma~\ref{lem: altitudes are audible}, $(p_1, p_2) = (q_1, q_2)$. By Viviani's theorem $p_1 + p_2 + p_3 = q_1 + q_2 + q_3$ , so $p_3 = q_3$ as well. Thus $P$ and $Q$ have the same multiset of distances to the sides of $T$. Since the symmetry group of $T$ realizes every permutation of its three sides, $P$ and $Q$ are equal up to a symmetry of $T$.
\end{proof}

\begin{proposition} \label{prop: 45-45-90 triangle}
    One can always hear where a 45-45-90 triangle is struck, up to symmetry.
\end{proposition}

\begin{proof}
Let our domain $\Omega$ be the 45-45-90 triangle with vertices $(0,0)$, $(\frac{1}{2},0)$, and $(\frac{1}{2},\frac{1}{2})$. Let $L$ be the intersection of $\interior \Omega$ and the right-angle bisector of $\Omega$. Let $\Omega_0$ be the triangle with vertices $(0,0)$, $(\frac{1}{4},\frac{1}{4})$, and $(\frac{1}{2},0)$; see Figure~\ref{fig: 45-45-90 triangle 1}.

\begin{figure}[h]
    \centering
    % The 45-45-90 triangle Omega with sub-region Omega_0 (bottom-left half)
    % shaded red, divided from its mirror image by the dotted line L (the bisector
    % of the right angle / axis of symmetry).
    \begin{tikzpicture}[scale=9, every node/.style={font=\small}]
      \coordinate (A) at (0,0);
      \coordinate (B) at (0.5,0);
      \coordinate (C) at (0.5,0.5);
      \coordinate (M) at (0.25,0.25); % midpoint of hypotenuse = far end of L

      % Omega_0 : bottom-left half, shaded red
      \fill[red!15] (A) -- (B) -- (M) -- cycle;

      % the triangle Omega
      \draw[thick] (A) -- (B) -- (C) -- cycle;

      % dividing line L (blue, dotted)
      \draw[blue, dotted, thick] (B) -- (M);
      \node[blue] at (0.40,0.145) {$L$};

      % region label
      \node[red!70!black] at (0.26,0.06) {$\interior \Omega_0$};
    \end{tikzpicture}
    \caption{The region $\interior \Omega_0$ is shaded red, and the bisector $L$ of the right angle is the blue dotted line.}
    \label{fig: 45-45-90 triangle 1}
\end{figure}
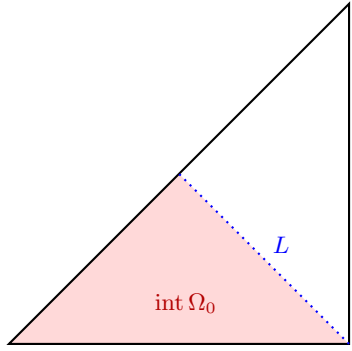

The reflection $\phi \colon \Omega \to \Omega$ across $L$, defined by $\phi(x_1,x_2) = (\frac{1}{2}-x_2, \frac{1}{2}-x_1)$, is an isometry of $\Omega$. Since isometries preserve the local Weyl counting function, it suffices to show uniqueness within $\Omega_0$.

Let $p, q \in \interior \Omega_0 \cup L$ be distinct points with $N(p,\lambda) = N(q,\lambda)$ for all $\lambda \geq 0$. Write $p = (p_1,p_2)$ and $q = (q_1,q_2)$, and introduce the substitutions
\[
    a = \cos(2\pi p_1), \quad b = \cos(2\pi p_2), \quad c = \cos(2\pi q_1), \quad d = \cos(2\pi q_2).
\]

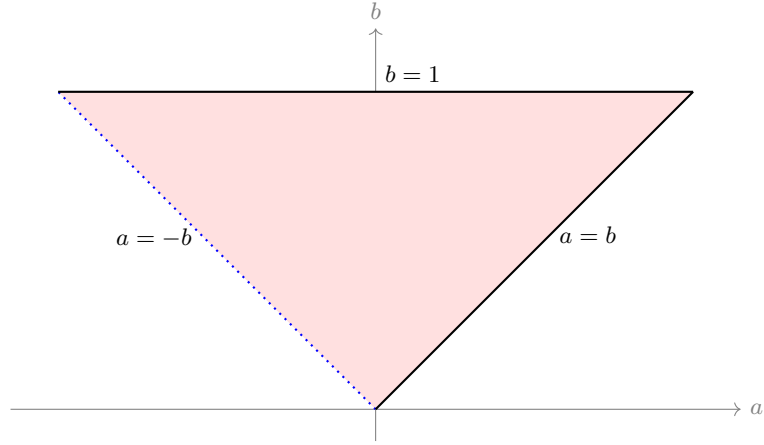
\begin{figure}[h]
    \centering
    % Parameter-space picture. Under a=cos(2*pi*p1), b=cos(2*pi*p2), the closed
    % sub-region Omega_0 maps to the triangle { |a| <= b, 0<b<1 } in the (a,b)-plane.
    % Edge a=b <-> hypotenuse; edge a=-b <-> the right-angle bisector L; b=1 <-> leg.
    \begin{tikzpicture}[scale=4.2, every node/.style={font=\small}]
      % axes
      \draw[->,gray] (-1.15,0) -- (1.15,0) node[right]{$a$};
      \draw[->,gray] (0,-0.1) -- (0,1.2) node[above]{$b$};

      % the region |a| <= b, 0<b<1
      \fill[red!12] (0,0) -- (1,1) -- (-1,1) -- cycle;

      % edges
      \draw[blue, dotted, thick] (0,0) -- (-1,1); % a=-b : bisector L
      \draw[thick] (0,0) -- (1,1);               % a= b : hypotenuse
      \draw[thick] (-1,1) -- (1,1);              % b= 1 : leg

      % boundary labels (equations only)
      \node[below left]  at (-0.55,0.60) {$a=-b$};
      \node[right] at ( 0.55,0.55) {$a=b$};
      \node[above right] at (0,1.0)      {$b=1$};
    \end{tikzpicture}
    \caption{The image of $\Omega_0$ in the $(a,b)$-plane under $a=\cos(2\pi p_1)$, $b=\cos(2\pi p_2)$: the triangle $|a|\le b$, $0<b<1$. The edges $a=b$ (hypotenuse) and $b=1$ (leg $x_2=0$) lie in $\partial\Omega$, while the blue dotted edge $a=-b$ is the bisector $L$.}
    \label{fig: 45-45-90 parameter space}
\end{figure}

Within $\Omega_0$, this substitution is one-to-one.  Moreover $a,c \in (-1,1)$, $b,d \in (0,1)$, $|a| \le b$, and $|c| \le d$, with equality $|a|=b$ (equivalently $a=-b$) precisely when $p \in L$, and likewise $|c|=d$ precisely when $q \in L$.

Using the eigenfunctions from Lemma~\ref{lem: 45-45-90 eigenfunction}, we note that $|e_{(n_1,n_2)}(p)|^2 = |\varphi_{(n_1,n_2)}(p)|^2$ depends only on $\{n_1^2, n_2^2\}$ as a set, so for any $\lambda$ the sum $N^*(p,\lambda) := \sum_{\lambda_\xi = \lambda}|e_\xi(p)|^2$ receives equal contributions from all $\xi$ in the set $\{(\pm n_1, \pm n_2), (\pm n_2, \pm n_1)\}$.  For generic $\lambda$ this gives a multiplicity of $8$.

We work with four specific eigenvalues $\lambda_1 < \lambda_2 < \lambda_3 < \lambda_4$ associated to $\xi = (2,1), (3,1), (3,2), (4,2)$ respectively. By the computation in Lemma~\ref{lem: 45-45-90 N* computations}, for $\xi = (2,1)$ we have
\[
    N^*(p,\lambda_1) = 32(1-a^2)(1-b^2)(b-a)^2.
\]
Equating to $N^*(q,\lambda_1)$ gives our first equation:
\begin{align}
    (1-a^2)(1-b^2)(b-a)^2 = (1-c^2)(1-d^2)(d-c)^2. \label{eq:E1}
\end{align}
Note the common factor $(1-a^2)(1-b^2)(b-a)^2$ is nonzero on $\Omega_0$ since $a \in (-1,1)$, $b \in (0,1)$, and $b > a$; likewise for $(1-c^2)(1-d^2)(d-c)^2$. Hence the divisions by \eqref{eq:E1} below are valid, including when $p$ or $q$ lies on $L$.

By the computation in Lemma~\ref{lem: 45-45-90 N* computations}, for $\xi = (3,1)$ we have
\[
    N^*(p,\lambda_2) = 128(1-a^2)(1-b^2)(b-a)^2(a+b)^2.
\]
Equating to $N^*(q,\lambda_2)$ and dividing by \eqref{eq:E1}:
\begin{align}
    (a+b)^2 = (c+d)^2. \label{eq:E2}
\end{align}
By the computation in Lemma~\ref{lem: 45-45-90 N* computations}, for $\xi = (3,2)$ we have
\[
    N^*(p,\lambda_3) = 32(1-a^2)(1-b^2)(b-a)^2(4ab+1)^2.
\]
Equating to $N^*(q,\lambda_3)$ and dividing by \eqref{eq:E1}:
\begin{align}
    (4ab+1)^2 = (4cd+1)^2. \label{eq:E3}
\end{align}
We now split into two cases according to whether the points lie on $L$. Equation \eqref{eq:E2} forces $a+b = 0$ if and only if $c+d = 0$. Thus either both points lie on $L$, or both lie in $\interior \Omega_0$.

\medskip
\noindent\textbf{Case 1: $p$ and $q$ lie in the interior of $\Omega_0$.} Here $|a| < b$ and $|c| < d$, so $a+b > 0$ and $c+d > 0$. By the computation in Lemma~\ref{lem: 45-45-90 N* computations}, for $\xi = (4,2)$ we have
\[
    N^*(p,\lambda_4) = 2048(1-a^2)a^2(1-b^2)b^2(b-a)^2(a+b)^2.
\]
Equating to $N^*(q,\lambda_4)$ and dividing by \eqref{eq:E1} and \eqref{eq:E2}:
\begin{align}
    a^2b^2 = c^2d^2. \label{eq:E4}
\end{align}
Expanding \eqref{eq:E3} and substituting in \eqref{eq:E4} yields
\begin{align}
    ab = cd. \label{eq:E5}
\end{align}
Expanding \eqref{eq:E2} and substituting in \eqref{eq:E5} yields
\begin{align}
    a^2+b^2 = c^2+d^2. \label{eq:E6}
\end{align}
Plugging in $b = cd/a$ gotten from \eqref{eq:E5} gives
\begin{align}
    a^4 - (c^2+d^2)a^2 + c^2d^2 = 0,
\end{align}
whose solutions are $a^2 = c^2$ or $a^2 = d^2$.

This yields eight candidate relationships between $a,b,c,d$. Checking against the constraints $ab = cd$, $b > 0$, $d > 0$, and the strict inequality $|a|+|c| < b+d$  eliminates all but $a = c$, $b = d$.  Since the substitution is one-to-one, we conclude $p = q$, a contradiction.

\medskip
\noindent\textbf{Case 2: $p$ and $q$ lie on $L$.} Here $a = -b$ and $c = -d$, so $a+b = c+d = 0$. In this case, we only use \eqref{eq:E1} and \eqref{eq:E3}. Substituting $b = -a$, $d = -c$ and writing $s = a^2$, $r = c^2 \in (0,1)$, equation \eqref{eq:E1} becomes
\[
    s(1-s)^2 = r(1-r)^2,
\]
while \eqref{eq:E3}, using $4ab+1 = 1-4a^2$, becomes
\[
    (1-4s)^2 = (1-4r)^2, \qquad \text{so} \qquad s = r \ \text{ or } \ s+r = \tfrac{1}{2}.
\]
In the latter case, writing $s = \tfrac{1}{4}+e$ and $r = \tfrac{1}{4}-e$, the first identity gives
\[
    s(1-s)^2 - r(1-r)^2 = 2e\bigl(\tfrac{3}{16}+e^2\bigr) = 0,
\]
which forces $e = 0$, hence $s = r$. On $L$ we have $a,c \in (-1,0)$, so $a = c$, $b = d$ and finally $p = q$, a contradiction.

In either case $p = q$, so no two distinct points of $\Omega_0$ share the same local Weyl counting function. By Lemma~\ref{lem: billiards}, one can always hear where a 45-45-90 triangle is struck, up to the symmetry $\phi$.
\end{proof}

\section{Proof of Theorem \ref{thm: main 2}} \label{sec: proof of surprising theorem}

\begin{proof}
Let $p = (\tfrac{7\sqrt3}{16}, \tfrac{3}{16})$ and $q = (\tfrac{5\sqrt3}{16}, \tfrac{3}{16})$, two distinct points in the interior of $\Omega$, not related by any isometry of $\Omega$. It suffices to show $|\varphi_\xi(p)| = |\varphi_\xi(q)|$ for every $\xi = A^{-t}n$, $n \in \Z^2$, where $\varphi_\xi$ is our calculated eigenfunction from Lemma~\ref{lem: 30-60-90 eigenfunction}. Fix such an $n$, set $u = n_1+n_2$ and $v = n_1-n_2$. Note $u \equiv v \pmod 2$ since $u-v=2n_2$.

Plugging in $p_1 = \tfrac{7\sqrt3}{16}$, $p_2 = q_2 = \tfrac{3}{16}$, and $q_1 = \tfrac{5\sqrt3}{16}$ into our eigenfunction yields
\begin{align*}
    \varphi_\xi(p) &= \sin\tfrac{5\pi u}{8}\sin\tfrac{3\pi v}{8} - \sin\tfrac{\pi u}{4}\sin\tfrac{\pi v}{2} - \sin\tfrac{7\pi u}{8}\sin\tfrac{\pi v}{8}, \\[4pt]
    \varphi_\xi(q) &= \sin\tfrac{\pi u}{2}\sin\tfrac{\pi v}{4} - \sin\tfrac{\pi u}{8}\sin\tfrac{3\pi v}{8} - \sin\tfrac{5\pi u}{8}\sin\tfrac{\pi v}{8}.
\end{align*}
By Lemma~\ref{lem: 30-60-90 Laurent verification}, writing $z = e^{i\pi u/8}$ and $w = e^{i\pi v/8}$, these satisfy the Laurent polynomial identities
\begin{equation}\label{eq: 30-60-90 Laurent}
    \varphi_\xi(p) - \varphi_\xi(q) = \frac{G_0 \cdot M}{4z^7w^4}, \qquad \varphi_\xi(p) + \varphi_\xi(q) = \frac{G_0 \cdot H}{4z^7w^4},
\end{equation}
where
\begingroup\small
\begin{align*}
    G_0 &= (z^2-1)(w^2-1)(z-w)(z^3-w)(zw-1)(z^3w-1), \\
    M   &= z + z^3 + w + z^4w + zw^2 + z^3w^2, \\
    H   &= (z^2+1)(z+w)(zw+1).
\end{align*}
\endgroup
Therefore
\[
    \varphi_\xi(p)^2 - \varphi_\xi(q)^2 = \bigl(\varphi_\xi(p) - \varphi_\xi(q)\bigr)\bigl(\varphi_\xi(p) + \varphi_\xi(q)\bigr) = \frac{G_0^2 \cdot M \cdot H}{16\,z^{14}w^8}.
\]
It remains to show that $G_0^2 \cdot M \cdot H = 0$ at $z = e^{i\pi u/8}$, $w = e^{i\pi v/8}$. By Lemma~\ref{lem: 30-60-90 polynomial vanishing},
\[
    G_0 \cdot M \cdot H = (z^8 - w^8)\cdot C + (z^{16}-1)\cdot A
\]
as polynomials, for explicit $C,A \in \Z[z,w]$. Since $z,w$ are $16$th roots of unity, $z^{16}=w^{16}=1$. Moreover $z^8/w^8 = e^{i\pi(u-v)} = 1$ since $u-v$ is even, so $z^8=w^8$. Hence both terms on the right vanish, giving $G_0 \cdot M \cdot H = 0$.

For every $\xi \in \widehat \Lambda = A^{-t} \Z^2$,
\[
    \varphi_\xi(p)^2 - \varphi_\xi(q)^2 = \frac{G_0^2 \cdot MH}{16z^{14}w^8} = 0,
\]
hence $|\varphi_\xi(p)| = |\varphi_\xi(q)|$.  Summing $|\varphi_\xi(p)|^2 = |\varphi_\xi(q)|^2$ over all $\xi$ with $|2\pi\xi| \leq \lambda$ gives $N(p,\lambda) = N(q,\lambda)$ for all $\lambda \geq 0$. Since $p$ and $q$ are not related by any isometry of $\Omega$, this exhibits a pair of points on the 30-60-90 triangle that are audibly indistinguishable, as claimed.
% \textcolor{red}{[To do: Argue that $p$ and $q$ are the only such audibly indistinguishable pair of points.]}
\end{proof}

\appendix

%%%%%%%%%%%%%%%%%%%%%%%%%%%%%%%%%%%%%%%%%
% APPENDIX: Lattices and Fourier series %
%%%%%%%%%%%%%%%%%%%%%%%%%%%%%%%%%%%%%%%%%

\section{Lattices and Fourier series}\label{app: lattices and fourier series}

This section reviews some of the machinery we use to compute the Laplace eigenfunctions and their eigenvalues for the three special triangles. We refer the reader to \cite{steinFourierAnalysisIntroduction2003, wolffLecturesHarmonicAnalysis2003} for background on Fourier series and \cite{gruberGeometryNumbers1987} for lattices.

\subsection{Lattices}

We start by reviewing some facts about lattices, found in \cite{gruberGeometryNumbers1987}.

\begin{definition}\label{def: lattice}
    A \emph{lattice} in $\R^d$ is a subset of the form
    \[
        A\Z^d = \{An : n \in \Z^d \}
    \]
    where $A$ is a nonsingular $d \times d$ matrix with real entries.
\end{definition}

We have a helpful equivalent characterization of lattices below, which follows readily from \cite[\S1.3 Theorem 2]{gruberGeometryNumbers1987}. We let
\[
    B_r = \{x \in \R^d : |x| < r\}
\]
denote the open ball of radius $r$ centered at the origin.

\begin{proposition}\label{prop: lattice characterization}
    A subgroup $\Lambda$ of $\R^d$ is a lattice if and only if both of the following hold:
    \begin{enumerate}
        \item There exists $\epsilon > 0$ such that the balls $B_\epsilon + x$ for $x \in \Lambda$ are disjoint.
        \item There exists $r > 0$ such that the balls $B_r + x$ for $x \in \Lambda$ cover $\R^d$.
    \end{enumerate}
\end{proposition}

We will generally consider a lattice $\Lambda$ both as a set of points in $\R^d$, and also as a group acting on $\R^d$ by translation, i.e.
\[
    \begin{split}
        \Lambda \times \R^d &\to \R^d \\
        (x,y) &\mapsto x + y.
    \end{split}
\]
Considering $\Lambda$ as a group acting on $\R^d$, we describe two natural fundamental domains.

\begin{example}[Parallelepiped]
    Given a lattice $\Lambda = A \Z^d$, the parallelepiped $A([0,1]^d)$ is a fundamental domain for $\Lambda$.
\end{example}

\begin{example}[Voronoi cell]
    Given a lattice $\Lambda$ in $\R^d$, the Voronoi cell
    \[
        \{y \in \R^d : |y| \leq |x - y| \text{ for each $x \in \Lambda$}\}
    \]
    is a fundamental domain for $\Lambda$.
\end{example}

Note that both the parallelepiped and the Voronoi cell are compact, the latter from part (2) of Proposition \ref{prop: lattice characterization}.

Given a lattice $\Lambda$ in $\R^d$, we say a function $f$ on $\R^d$ is \emph{$\Lambda$-periodic} if
\[
    f(x + y) = f(x) \qquad \text{ for each } x \in \Lambda, \ y \in \R^d.
\]
For $1 \leq p \leq \infty$, by $L^p(\R^d/\Lambda)$, we mean the set of complex-valued $\Lambda$-periodic functions on $\R^d$ whose restriction to any fundamental domain $D$ of $\Lambda$ is in $L^p(D)$.

\begin{proposition}
    Suppose $f \in L^1(\R^d/\Lambda)$. Then,
    \[
        \int_{D} f(x) \, dx = \int_{D'} f(x) \, dx
    \]
    for any two fundamental domains $D$ and $D'$ of $\Lambda$.
\end{proposition}

\begin{proof}
    Since $D$ and $D'$ are fundamental domains,
    \[
        \sum_{x \in \Lambda} \1_{D + x} = 1 \qquad \text{ and } \qquad \sum_{x \in \Lambda} \1_{D' - x} = 1 \qquad \text{a.e. in $\R^d$}.
    \]
    It follows
    \begin{align*}
        \int_{D'} f(y) \, dy
        &= \int_{\R^d} f(y) \1_{D'}(y) \, dy \\
        &= \sum_{x \in \Lambda} \int_{\R^d} f(y) \1_{D'}(y) \1_{D + x}(y) \, dy \\
        &= \sum_{x \in \Lambda} \int_{\R^d} f(y) \1_{D'}(y) \1_{D}(y - x) \, dy \\
        &= \sum_{x \in \Lambda} \int_{\R^d} f(y + x) \1_{D'}(y + x) \1_{D}(y) \, dy && \text{(change of variables)} \\
        &= \sum_{x \in \Lambda} \int_{\R^d} f(y) \1_{D'}(y + x) \1_{D}(y) \, dy && \text{($f$ is $\Lambda$-periodic)} \\
        &= \sum_{x \in \Lambda} \int_{\R^d} f(y) \1_{D' - x}(y) \1_{D}(y) \, dy \\
        &= \int_{\R^d} f(y) \1_{D}(y) \, dy \\
        &= \int_{D} f(y) \, dy.
    \end{align*}
\end{proof}

For general $1 \leq p < \infty$ and for a lattice $\Lambda$ with fundamental domain $D$, we define
\[
    \|f\|_{L^p(\R^d/\Lambda)} := \|f\|_{L^p(D)} = \left( \int_D |f(x)|^p \, dx \right)^\frac 1 p,
\]
and note by the proposition above that this norm is indifferent to our choice of fundamental domain. In the case $p = 2$, $L^2(\R^d/\Lambda)$ is a Hilbert space with Hermitian inner product
\[
    \langle f,g \rangle_{L^2(\R^d/\Lambda)} = \int_D f(x) \overline{g(x)} \, dx.
\]

Key to Fourier series in higher dimensions is the notion of the dual lattice:

\begin{definition}[Dual lattice]
    Given a lattice $\Lambda = A \Z^d$ in $\R^d$, the \emph{dual lattice} is
    \[
        \widehat \Lambda = A^{-t} \Z^d,
    \]
    where $A^{-t}$ denotes the inverse transpose of $A$.
\end{definition}

One should think of the dual lattice $\widehat \Lambda$ to $\Lambda$ as the largest lattice for which $x \cdot \xi \in \Z$ for each $x \in \Lambda$ and $\xi \in \widehat \Lambda$.

\subsection{Fourier series with respect to a lattice}

Now we recall some facts on Fourier series in $\R^d$.

\begin{theorem}[Fourier series] \label{thm: 2d fourier series}
    For each $f \in L^2(\R^d/\Z^d)$,
    \[
        f(x) = \sum_{n \in \Z^d} e^{2\pi i x \cdot n} \widehat f(n),
    \]
    where equality and convergence are in the sense of $L^2(\R^d/\Z^d)$.
\end{theorem}

We require a mild generalization of the $d$-dimensional Fourier series formula to $\Lambda$-periodic functions for general lattices $\Lambda$ in $\R^d$.

\begin{corollary}\label{cor: lattice fourier series}
    Let $\Lambda$ be a lattice in $\R^d$. For each $f \in L^2(\R^d/\Lambda)$,
    \[
        f(x) = \sum_{\xi \in \widehat \Lambda} e^{2\pi i x \cdot \xi} \widehat f(\xi),
    \]
    where equality and convergence are in the sense of $L^2(\R^d/\Lambda)$.
\end{corollary}

\begin{proof}
    Write $\Lambda = A \Z^d$ for an invertible linear operator $A$ on $\R^d$. For $f \in L^2(\R^d/\Lambda)$, we note $g := f \circ A \in L^2(\R^d/\Z^d)$ and so we may apply the two-dimensional Fourier series formula. In the following calculation, $Q$ denotes the unit cube $[0,1]^d$ and note that $Q$ is a fundamental domain for $\Z^d$ and $AQ$ is a fundamental domain for $\Lambda$. We have
    \begin{align*}
        f(x) &= g(A^{-1}x) \\
        &= \sum_{n \in \Z^d} e^{2\pi i A^{-1}x \cdot n} \widehat g(n) \\
        &= \sum_{n \in \Z^d} e^{2\pi i A^{-1}x \cdot n} \int_Q e^{-2\pi i y \cdot n} g(y) \, dy \\
        &= \sum_{n \in \Z^d} e^{2\pi i A^{-1}x \cdot n} \int_Q e^{-2\pi i y \cdot n} f(Ay) \, dy \\
        &= \sum_{n \in \Z^d} e^{2\pi i A^{-1}x \cdot n} |\det A|^{-1} \int_{AQ} e^{-2\pi i A^{-1}y \cdot n} f(y) \, dy \\
        &= \sum_{n \in \Z^d} e^{2\pi i x \cdot A^{-t}n } \frac{1}{|AQ|}  \int_{AQ} e^{-2\pi i y \cdot A^{-t} n} f(y) \, dy \\
        &= \sum_{\xi \in \widehat \Lambda} e^{2\pi i x \cdot \xi} \frac{1}{|AQ|} \int_{AQ} e^{-2\pi i y \cdot \xi} f(y) \, dy \\
        &= \sum_{\xi \in \widehat \Lambda} e^{2\pi i x \cdot \xi} \widehat f(\xi)
    \end{align*}
    as desired.
\end{proof}

%%%%%%%%%%%%%%%%%%%%%%%%%%%%%%%%%%%%
% APPENDIX: Auxiliary computations %
%%%%%%%%%%%%%%%%%%%%%%%%%%%%%%%%%%%%

\section{Auxiliary computations}\label{sec: trig computations}

This section collects longer, routine computations that are used to simplify or verify expressions appearing in the proofs of Section \ref{sec: proof of main theorem}, so that those proofs can focus on the underlying argument rather than on algebraic manipulation.

\begin{lemma}\label{lem: 45-45-90 N* computations}
    Let $\Omega$ be the 45-45-90 triangle in standard position, let $p = (p_1,p_2)$ be a point in the interior of $\Omega$, and set
    \[
        a = \cos(2\pi p_1), \qquad b = \cos(2\pi p_2).
    \]
    Using the eigenfunctions $e_\xi$ from Lemma \ref{lem: 45-45-90 eigenfunction} and writing $N^*(p,\lambda) = \sum_{\lambda_\xi = \lambda} |e_\xi(p)|^2$, we have:
    \begin{enumerate}
        \item For $\xi = (2,1)$,
        \[
            N^*(p,\lambda_1) = 32(1-a^2)(1-b^2)(b-a)^2.
        \]
        \item For $\xi = (3,1)$,
        \[
            N^*(p,\lambda_2) = 128(1-a^2)(1-b^2)(b-a)^2(a+b)^2.
        \]
        \item For $\xi = (3,2)$,
        \[
            N^*(p,\lambda_3) = 32(1-a^2)(1-b^2)(b-a)^2(4ab+1)^2.
        \]
        \item For $\xi = (4,2)$,
        \[
            N^*(p,\lambda_4) = 2048(1-a^2)a^2(1-b^2)b^2(b-a)^2(a+b)^2.
        \]
    \end{enumerate}
\end{lemma}

\begin{proof}
    Throughout, recall $N^*(p,\lambda) = 8\,|e_\xi(p)|^2$ for the relevant $\xi$, since the multiplicity is $8$ as noted above.

    \medskip
    \noindent\textbf{(1) $\xi = (2,1)$.} Using $\sin(4\pi t) = 2\sin(2\pi t)\cos(2\pi t)$,
    \begin{align*}
        N^*(p,\lambda_1) &= 8\left[\sin(4\pi p_2)\sin(2\pi p_1) - \sin(4\pi p_1)\sin(2\pi p_2)\right]^2 \\
        &= 8\left[2\sin(2\pi p_2)\cos(2\pi p_2)\sin(2\pi p_1) - 2\sin(2\pi p_1)\cos(2\pi p_1)\sin(2\pi p_2)\right]^2 \\
        &= 32\left[\sin(2\pi p_1)\sin(2\pi p_2)(\cos(2\pi p_2) - \cos(2\pi p_1))\right]^2 \\
        &= 32(1-a^2)(1-b^2)(b-a)^2.
    \end{align*}

    \medskip
    \noindent\textbf{(2) $\xi = (3,1)$.} Using $\sin(6\pi t) = 3\sin(2\pi t) - 4\sin^3(2\pi t)$,
    \begin{align*}
        N^*(p,\lambda_2) &= 8\left[(3\sin(2\pi p_2)-4\sin^3(2\pi p_2))\sin(2\pi p_1) - (3\sin(2\pi p_1)-4\sin^3(2\pi p_1))\sin(2\pi p_2)\right]^2 \\
        &= 8\left[\sin(2\pi p_1)\sin(2\pi p_2)\left(\sin^2(2\pi p_1) - \sin^2(2\pi p_2)\right)\cdot 4\right]^2 \\
        &= 128(1-a^2)(1-b^2)(b^2-a^2)^2 \\
        &= 128(1-a^2)(1-b^2)(b-a)^2(a+b)^2.
    \end{align*}

    \medskip
    \noindent\textbf{(3) $\xi = (3,2)$.} Using $\sin(4\pi t) = 2\sin(2\pi t)\cos(2\pi t)$ and $\sin(6\pi t) = 3\sin(2\pi t) - 4\sin^3(2\pi t)$,
    \begin{align*}
        N^*(p,\lambda_3) &= 8\Big[2(3\sin(2\pi p_2)-4\sin^3(2\pi p_2))\sin(2\pi p_1)\cos(2\pi p_1) \\
        &\qquad - 2(3\sin(2\pi p_1)-4\sin^3(2\pi p_1))\sin(2\pi p_2)\cos(2\pi p_2)\Big]^2 \\
        &= 32\left[\sin(2\pi p_1)\sin(2\pi p_2)\left((3-4\sin^2(2\pi p_2))\cos(2\pi p_1) - (3-4\sin^2(2\pi p_1))\cos(2\pi p_2)\right)\right]^2 \\
        &= 32(1-a^2)(1-b^2)(b-a)^2(4ab+1)^2.
    \end{align*}

    \medskip
    \noindent\textbf{(4) $\xi = (4,2)$.} Using $\sin(8\pi t) = 2\sin(4\pi t)\cos(4\pi t)$ and $\sin(4\pi t) = 2\sin(2\pi t)\cos(2\pi t)$,
    \begin{align*}
        N^*(p,\lambda_4) &= 8\left[\sin(4\pi p_1)\sin(8\pi p_2) - \sin(8\pi p_1)\sin(4\pi p_2)\right]^2 \\
        &= 8\left[16\,\sin(2\pi p_1)\cos(2\pi p_1)\,\sin(2\pi p_2)\cos(2\pi p_2)\left(\cos^2(2\pi p_2)-\cos^2(2\pi p_1)\right)\right]^2 \\
        &= 2048(1-a^2)a^2(1-b^2)b^2(b-a)^2(a+b)^2. \qedhere
    \end{align*}
\end{proof}

\begin{lemma}\label{lem: 30-60-90 Laurent verification}
    With $\varphi_\xi(p)$, $\varphi_\xi(q)$, $u$, $v$, $z$, $w$, $G_0$, $M$, $H$ as in the proof of Theorem \ref{thm: main 2},
    \[
        \varphi_\xi(p) - \varphi_\xi(q) = \frac{G_0 \cdot M}{4z^7w^4}, \qquad \varphi_\xi(p) + \varphi_\xi(q) = \frac{G_0 \cdot H}{4z^7w^4}.
    \]
\end{lemma}

\begin{proof}
    Writing $\sin(\tfrac{k\pi u}{8}) = (z^k-z^{-k})/(2i)$ and $\sin(\tfrac{k\pi v}{8}) = (w^k-w^{-k})/(2i)$ and expanding, $4z^7w^4\varphi_\xi(p)$ and $4z^7w^4\varphi_\xi(q)$ are each sums of $12$ monomials $z^aw^b$ with coefficients $\pm1$. Adding and subtracting gives $4z^7w^4(\varphi_\xi(p)\mp\varphi_\xi(q))$, each a sum of $24$ monomials with coefficients $\pm 1$. Expanding $G_0\cdot M$ and $G_0 \cdot H$ similarly gives sums of $24$ monomials with coefficients $\pm1$. The tables below, indexed by $(a,b) = (\deg_z,\deg_w)$, record every coefficient on both sides:

    \medskip
    \noindent\textbf{Difference, $4z^7w^4(\varphi_\xi(p)-\varphi_\xi(q))$ vs.\ $G_0\cdot M$:}
    \begin{center}
    \begin{tabular}{c|c||c|c||c|c}
        $(a,b)$ & coeff. & $(a,b)$ & coeff. & $(a,b)$ & coeff. \\\hline
        $(0,3)$ & $+1$ & $(5,0)$ & $+1$ & $(11,2)$ & $-1$ \\
        $(0,5)$ & $-1$ & $(5,8)$ & $-1$ & $(11,6)$ & $+1$ \\
        $(2,1)$ & $-1$ & $(6,1)$ & $-1$ & $(12,1)$ & $+1$ \\
        $(2,3)$ & $-1$ & $(6,7)$ & $+1$ & $(12,3)$ & $+1$ \\
        $(2,5)$ & $+1$ & $(8,1)$ & $+1$ & $(12,5)$ & $-1$ \\
        $(2,7)$ & $+1$ & $(8,7)$ & $-1$ & $(12,7)$ & $-1$ \\
        $(3,2)$ & $+1$ & $(9,0)$ & $-1$ & $(14,3)$ & $-1$ \\
        $(3,6)$ & $-1$ & $(9,8)$ & $+1$ & $(14,5)$ & $+1$
    \end{tabular}
    \end{center}

    \medskip
    \noindent\textbf{Sum, $4z^7w^4(\varphi_\xi(p)+\varphi_\xi(q))$ vs.\ $G_0\cdot H$:}
    \begin{center}
    \begin{tabular}{c|c||c|c||c|c}
        $(a,b)$ & coeff. & $(a,b)$ & coeff. & $(a,b)$ & coeff. \\\hline
        $(0,3)$ & $+1$ & $(5,0)$ & $+1$ & $(11,2)$ & $+1$ \\
        $(0,5)$ & $-1$ & $(5,8)$ & $-1$ & $(11,6)$ & $-1$ \\
        $(2,1)$ & $-1$ & $(6,1)$ & $+1$ & $(12,1)$ & $+1$ \\
        $(2,3)$ & $+1$ & $(6,7)$ & $-1$ & $(12,3)$ & $-1$ \\
        $(2,5)$ & $-1$ & $(8,1)$ & $-1$ & $(12,5)$ & $+1$ \\
        $(2,7)$ & $+1$ & $(8,7)$ & $+1$ & $(12,7)$ & $-1$ \\
        $(3,2)$ & $-1$ & $(9,0)$ & $-1$ & $(14,3)$ & $-1$ \\
        $(3,6)$ & $+1$ & $(9,8)$ & $+1$ & $(14,5)$ & $+1$
    \end{tabular}
    \end{center}
    Since the tables match term by term, $4z^7w^4(\varphi_\xi(p)-\varphi_\xi(q)) = G_0 \cdot M$ and $4z^7w^4(\varphi_\xi(p)+\varphi_\xi(q)) = G_0 \cdot H$. Dividing by $4z^7w^4$ gives the result.
\end{proof}

\begin{lemma}\label{lem: 30-60-90 polynomial vanishing}
    With $G_0$, $M$, $H$ as in the proof of Theorem \ref{thm: main 2}, $G_0 \cdot M \cdot H = (z^8-w^8)\cdot C + (z^{16}-1)\cdot A$, where
    \begin{align*}
        C &= z^2(z^2-1)(z^2+1)\bigl(-w^2z^6 - w^2z^4 + wz^9 + wz + z^{10}+z^8+z^6+z^4+z^2+1\bigr), \\
        A &= -(z-w)(z+w)\bigl(w^5z + w^4z^2 + w^4 + w^3z^3 + w^2z^4 - w^2 + wz^5 - wz + z^6 + z^4\bigr).
    \end{align*}
\end{lemma}

\begin{proof}
    Expanding both sides as polynomials in $z,w$ gives, in each case, a sum of $36$ monomials $z^aw^b$ with coefficients $\pm1$. The table below, indexed by $(a,b)=(\deg_z,\deg_w)$, records the coefficient of $z^aw^b$ on both sides:
    \begin{center}
    \begin{tabular}{c|c||c|c||c|c}
        $(a,b)$ & coeff. & $(a,b)$ & coeff. & $(a,b)$ & coeff. \\\hline
        $(0,4)$ & $+1$ & $(6,0)$ & $+1$ & $(14,2)$ & $+1$ \\
        $(0,6)$ & $-1$ & $(6,10)$ & $-1$ & $(14,8)$ & $-1$ \\
        $(1,3)$ & $+1$ & $(7,1)$ & $+1$ & $(15,1)$ & $+1$ \\
        $(1,7)$ & $-1$ & $(7,9)$ & $-1$ & $(15,9)$ & $-1$ \\
        $(2,2)$ & $-1$ & $(8,0)$ & $+1$ & $(16,2)$ & $+1$ \\
        $(2,4)$ & $+1$ & $(8,10)$ & $-1$ & $(16,4)$ & $-1$ \\
        $(2,6)$ & $-1$ & $(10,0)$ & $-1$ & $(16,6)$ & $+1$ \\
        $(2,8)$ & $+1$ & $(10,10)$ & $+1$ & $(16,8)$ & $-1$ \\
        $(3,1)$ & $-1$ & $(11,1)$ & $-1$ & $(17,3)$ & $-1$ \\
        $(3,9)$ & $+1$ & $(11,9)$ & $+1$ & $(17,7)$ & $+1$ \\
        $(4,2)$ & $-1$ & $(12,0)$ & $-1$ & $(18,4)$ & $-1$ \\
        $(4,8)$ & $+1$ & $(12,10)$ & $+1$ & $(18,6)$ & $+1$
    \end{tabular}
    \end{center}
    Since the coefficients agree term by term on both sides, $G_0 \cdot M \cdot H = (z^8-w^8)C + (z^{16}-1)A$.
\end{proof}

\bibliographystyle{alpha}
\bibliography{references}
\end{document}